\documentclass[a4paper,11pt,leqno]{amsart}
\usepackage{amsfonts}
\usepackage{amssymb}
\usepackage{amsmath}
\usepackage{graphicx}
\usepackage{wrapfig}
\usepackage{url}
\usepackage{hyperref}
\hypersetup{
	colorlinks,
	linkcolor = {blue},
	citecolor = {blue},
	filecolor = {blue},
	urlcolor = {blue} 
} 

\newtheorem{theorem}{Theorem}[section]
\newtheorem{cor}{Corollary}[section]
\newtheorem{pr}{Problem}
\newtheorem{lemma}{Lemma}[section]

\newtheorem{rem}{Remark}[section]

\newcommand\blfootnote[1]{%
  \begingroup
  \renewcommand\thefootnote{}\footnote{#1}%
  \addtocounter{footnote}{-1}%
  \endgroup
}

\begin{document}
\title{}

\markboth{Nguyen Tran Huu Thinh}  {From two simple problems to the connection of special points}

\begin{minipage}[0.8 \textheight]{0.8 \textwidth}
 \begin{flushleft}
 \tiny{\normalsize INTERNATIONAL JOURNAL OF GEOMETRY}\\
 \end{flushleft}
 \tiny{\large Vol. 3 (2013), No. 1, 5 - 14}
\end{minipage}
\includegraphics[height=1.0 cm]{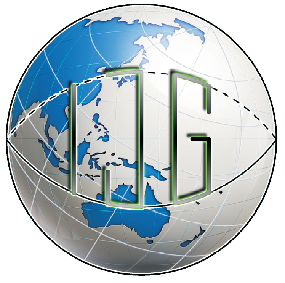} \centerline{}\bigskip

\centerline{}\bigskip

\centerline {\Large{\bf FROM TWO SIMPLE PROBLEMS TO }}\centerline{\Large{\bf  THE CONNECTION OF SPECIAL POINTS}}%
\centerline {\Large{\bf }}

\bigskip

\begin{center}
{\large NGUYEN TRAN HUU THINH}

\centerline{}
\end{center}

\bigskip

\textbf{Abstract.} Both the USA TST 2008 \cite{usatst2008} and the ELMO Shortlist 2013 \cite{elmoshortlist2013} suggested two issues that are connected to fixed points. These problems provide a strong linkage between the various attributes of specific points in a triangle. In this article, we will first investigate various theorems concerning the fixed points that have been presented, and then we will demonstrate how those points are connected to a few triangle centers.\blfootnote{\textbf{Keywords and phrases:} Fixed point, Triangle, Brocard point, Hagge circle} \blfootnote{\textbf{(2020)Mathematics Subject Classification: }51M04, 51-08} \blfootnote{Received: 1.01.2013. \ In revised form: 13.06.2013. \ Accepted:
10.09.2013.}
\bigskip

\section{Introduction}
The fixed-point problems have been seen in a great number of international mathematical olympiad competitions \cite{fix01}. The generalization of these problems in this area, as well as the features of these problems, leave behind a plethora of notable results \cite{nvlfix}. In this work, we talk about some of those problems and show the relationship to a few of different kinds of unique triangle centers.

In USA Team Selection Test 2008 \cite{usatst2008}, there is a fixed-point problem as following:
\begin{pr}
\label{problem_usatst2008}
Given triangle $ABC$, centroid $G$. Let $P$ be a variable point on segment $BC$. Let $Q, R$ be points on $AC, AB$ such that $PQ \parallel AB, PR \parallel AC$. Prove that $(AQR)$ passes through a fixed point $X$ such that $(AB, AG)=(AX, AC)$.
\end{pr}

\begin{center}
\begin{figure}[htbp]
\includegraphics[scale=.5]{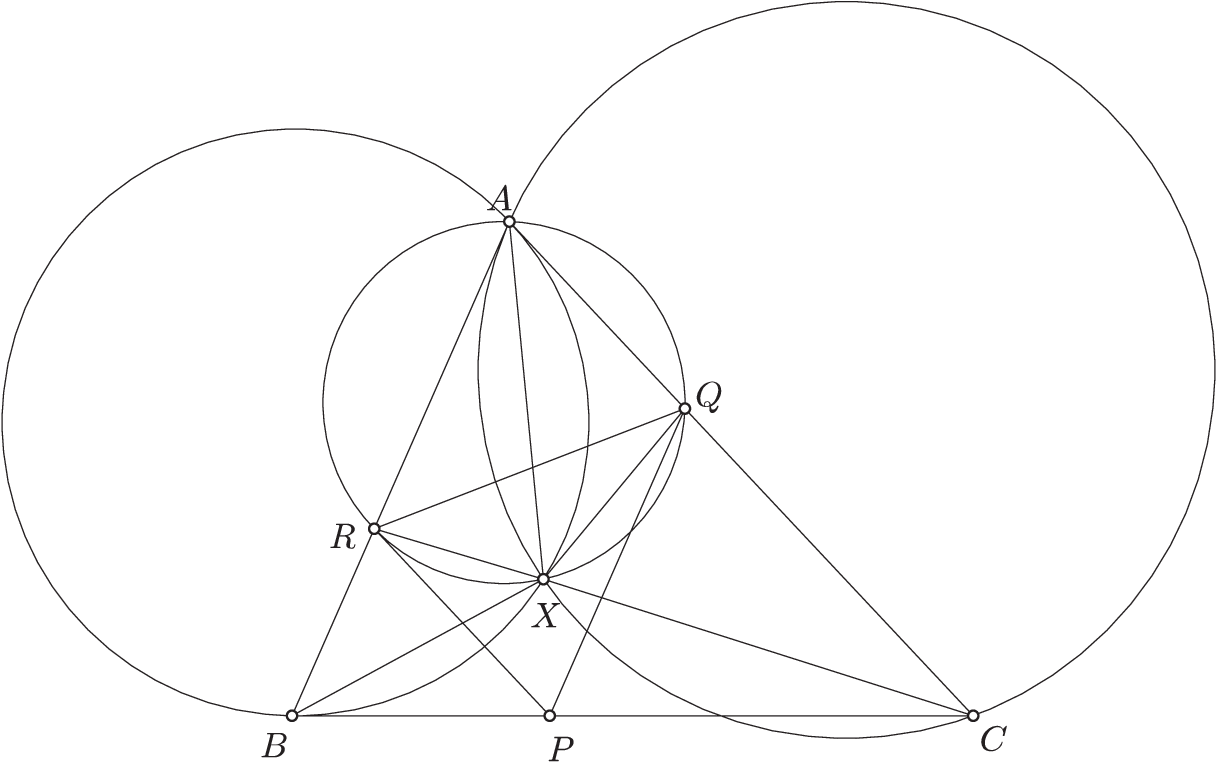}
\caption{Proof of Problem \ref{problem_usatst2008}}
\end{figure}
\end{center}

\begin{proof}
Let $\omega_1$ be the circle which passes $A, B$ and tangents to $AC$. Let $\omega_2$ be the circle which passes $A, C$ and tangents to $AB$. Let $X$ be intersection of $\omega_1$ and $\omega_2$, so $X$ is fixed. Since two triangle $ABX$ and $AXC$ are homothetic, hence $\frac{AX}{BX}=\frac{AC}{AB}$. On the other hand, $\frac{AC}{AB}=\frac{PR}{BR}=\frac{AQ}{BR}$ so $\frac{AX}{BX}=\frac{AQ}{BR}$, also $(BR, BX)=(AX, AX)$ hence two triangle $BRX$ and $AQX$ are homothetic, therefore $(RB, RX)=(QA, QX)$, following that $A, Q, X, R$ are cyclic. Moreover, we have

$$\frac{\sin (AB, AX)}{\sin (CA, CX)}=\frac{\sin (AB, AX)}{\sin (BA, BX)}=\frac{BX}{AX}=\frac{AB}{AC}$$
so $AX$ is the symmedian line of triangle $ABC$, which means $(AB, AG)=(AX, AC)$. The problem has been solved.
\end{proof}

With this configuration, ELMO Shortlist 2013 \cite{elmoshortlist2013} provide the following problem using the antiparallel lines instead of parallel ones:

\begin{pr}
\label{problem_elmo}
Given triangle $ABC$ and a centroid $G$. Let $P$ be a variable point on segment $BC$. Let $Q, R$ respectively be points on $AC, AB$ such that $PQ, PR$ are antiparallel lines of $AB, AC$. Prove that $(AQR)$ passes through a fixed point.
\end{pr}

\begin{center}
\begin{figure}[htbp]
\includegraphics[scale=.5]{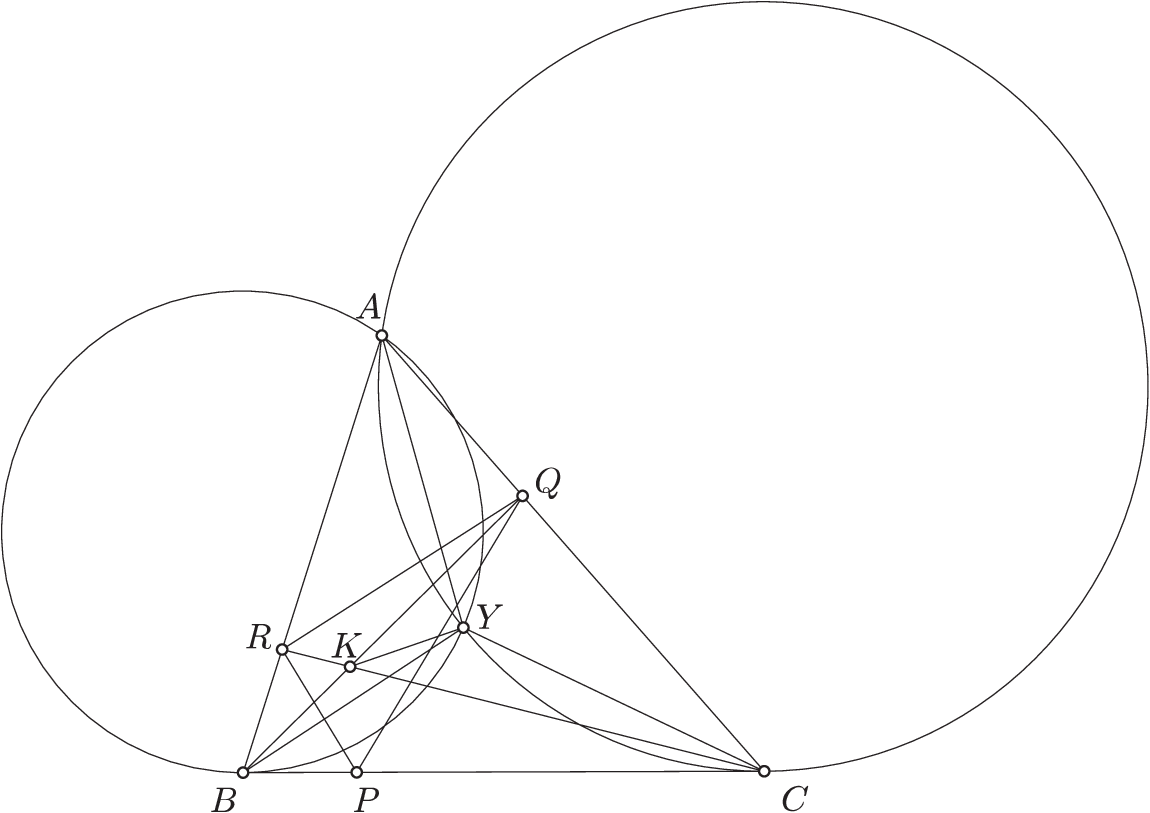}
\caption{Proof of Problem \ref{problem_elmo}}
\end{figure}
\end{center}

\begin{proof}
Let $\gamma_1$ be the circles passes $A, B$ and tangents to $BC$. Let $\gamma_2$ be the circle passes $A, C$ and tangents to $BC$. Let $Y$ be intersection of $\gamma_1$ and $\gamma_2$, so $Y$ is fixed. Let $K$ be intersection of $BQ$ and $CR$. We hav
$$(RK, RA)=(PC, PA)=(PB, PA)=(QK, QA) \text{ (mod } \pi).$$
Hence four points $A, R, K, Q$ are cyclic. Therefore:
$$\begin{aligned} (KB, KC)=(KR, KQ)=(AR, AQ) & =(AR, AY)+(AY, AQ) \\ & =(BY, BC)+(CB, CY) \\ & =(YB, YC) \text{ (mod } \pi). \end{aligned}$$
So four points $B, K, Y, C$ are cyclic. We have
$$(KR, KY)=(KY, KC)=(BY, BC)=(AR, AY) \text{ (mod } \pi).$$
Hence $Y$ is on $(ARK) \equiv (AQR)$. Moreover, by using the power of point's properties, we have the radical axis of $\gamma_1$ and $\gamma_2$ passes their common tangents, following that $AY$ passes $BC$ at its midpoint or $AY$ is the median line of triangle $ABC$.
\end{proof}
The subsequent sections will go into the various qualities associated with the aforementioned pair of fixed points, which are closely linked to several notable triangle centers.

\section{Properties related to the first problem}
Two above problems appeared in many contests \cite{usatst2008, elmoshortlist2013} and journals \cite{fix01} in the world, but it seems like there are just few properties about. This article will put forward some results about two fixed points mentioned above. We denotes $A_X, A_Y$ be that two points in the previous problems.
\begin{theorem}
\label{theorem_AX_project}
$A_X$ is projection point of the circumcenter of triangle $ABC$ on the $A-$symmedian line of that triangle.
\end{theorem}

\begin{center}
\begin{figure}[htbp]
\includegraphics[scale=.5]{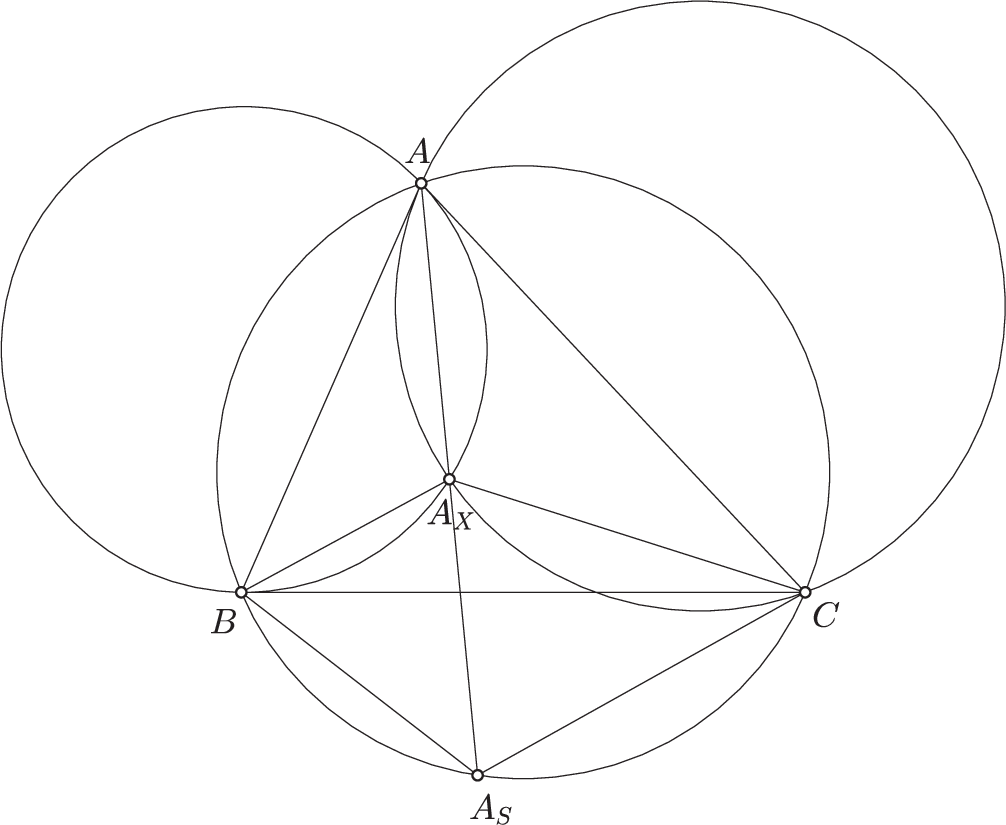}
\caption{Proof of Theorem \ref{theorem_AX_project}}
\end{figure}
\end{center}

\begin{proof}
Let $A_S$ be the second intersection of $AA_X$ and $(ABC)$. From Problem \ref{problem_usatst2008}, we have two triangle $ABA_X$ and $AA_XC$ are homothetic hence $A_XA^2=A_XB\cdot A_XC$. Moreover, we have
$$(A_XB, A_XA_S)=(A_XB, AA_X)=(AA_X, A_XC)=(A_XA_S, A_XC),$$
$$\begin{aligned} (A_SB, A_SA_X)=(CB, CA) & =(CB, CA_X)+(CA_X, CA) \\ & =(CB, CA_X)+(AA_S, AB) \\ & =(CB, CA_X)+(CA_S, CB) \\ & =(CA_S, CA_X). \end{aligned}$$
So two triangle $A_XBA_S$ and $A_XA_SC$ are homothetic. Therefore $A_XA_S^2=A_XB\cdot A_XC$. Hence $A_XA=A_XA_S$ or $A_X$ is midpoint of $AA_S$, following that $A_X$ is projection point of the circumcenter of triangle $ABC$ on the $A$-symmedian line of that triangle.
\end{proof}
\begin{theorem}
\label{theorem_2nd_brocard}
Brocard circle \cite{brocardcircle} is the circle which passes two Brocard points $Z_1, Z_2$ \cite{1stbrocard, 2ndbrocard}, Lemoine point $L$ \cite{lemoine} and circumcenter $O$ of triangle $ABC$. Let $A_X$ be the fixed point mentioned in Problem \ref{problem_usatst2008}. The construction of $B_X$ and $C_X$ is performed in a similar manner, but with the utilization of vertices $B$ and $C$. Then $A_X$, $B_X$ and $C_X$ lie on Brocard circle \cite{brocardcircle} of triangle $ABC$.
\end{theorem}

\begin{rem}
$A_XB_XC_X$ is the second Brocard triangle of triangle $ABC$.
\end{rem}

\begin{lemma}
\label{lemma_mannheim}
{\bf (Mannheim's theorem)} Given triangle $ABC$ and three points $L$, $M$, $N$ respectively be on $BC$, $CA$, $AB$. Let $A'$, $B'$, $C'$ respectively be three points that lie on $(AMN)$, $(BNL)$, $(CLM)$ such that $AA'$, $BB'$, $CC'$ are concurrent at $K$. Prove that four points $A'$, $B'$, $C'$, $K$ are cyclic.
\end{lemma}

\begin{center}
\begin{figure}[htbp]
\includegraphics[scale=.5]{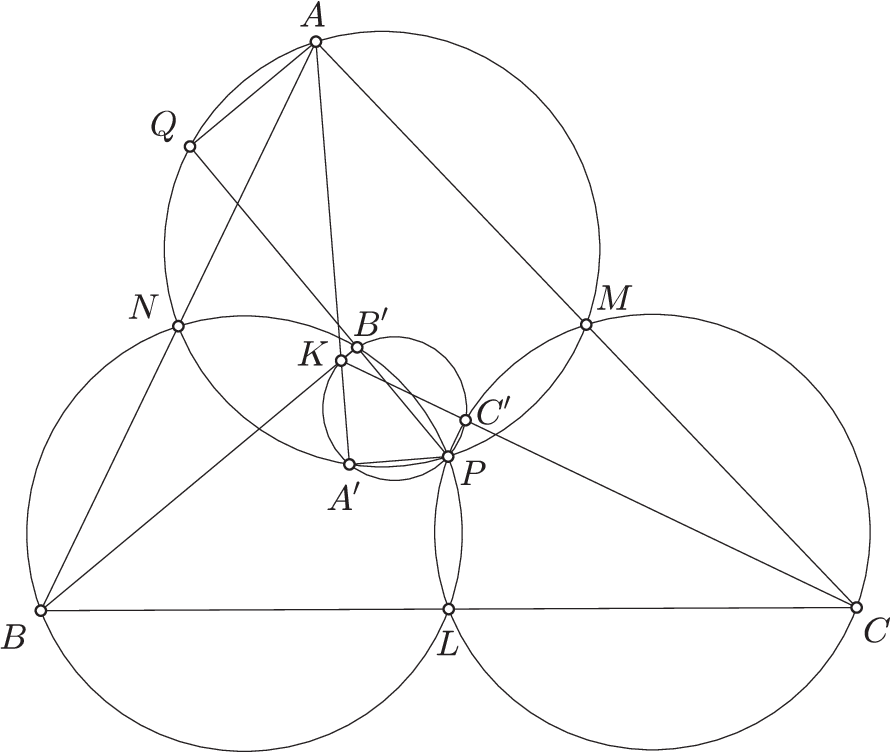}
\caption{Proof of Lemma \ref{lemma_mannheim}}
\end{figure}
\end{center}

\begin{proof}
Let $P$ be Miquel point of triangle $LMN$ with respect to triangle $ABC$. Let $Q$ be intersection of $PB'$ and $(AMN)$. Applying Reim's theorem to $(AMN)$ and $(BNL)$ we have $AQ \parallel BB'$. Through angle chasing:
$$(KA, KB')=(AA', AQ)=(PA', PB') \text{ (mod } \pi)$$
Therefore $A', P, B', K$ are cyclic. Similarly, $B', P, C', K$ are cyclic. Hence four points $A'$, $B'$, $C'$, $K$ are cyclic.
\end{proof}

\begin{lemma}
\label{lemma_XYZ}
Given triangle $ABC$ inscribed in $(O)$ and a point $P$ is not this triangle's vertex. Lines $AP$, $BP$, $CP$ respectively cut $(ABC)$ at $S$, $T$, $U$. Let $X$, $Y$, $Z$ be midpoints of segments $AS$, $BT$, $CU$, respectively. Prove that $X$, $Y$, $Z$ lie on $(OP)$ (a circle with center $O$ and radius $OP$).
\end{lemma}

\begin{proof}
We have $X$, $Y$,$ Z$ respectively be projection of $O$ on $AS$, $BT$, $CU$. Through angle chasing:
$$(XP, XO)=(YP, YO)=(ZP, ZO)=\frac{\pi}{2}.$$
Therefore $X, Y, Z$ lie on $(OP)$. The lemma has been proved.
\end{proof}

Back to Theorem \ref{theorem_2nd_brocard},

\begin{center}
\begin{figure}[htbp]
\includegraphics[scale=.5]{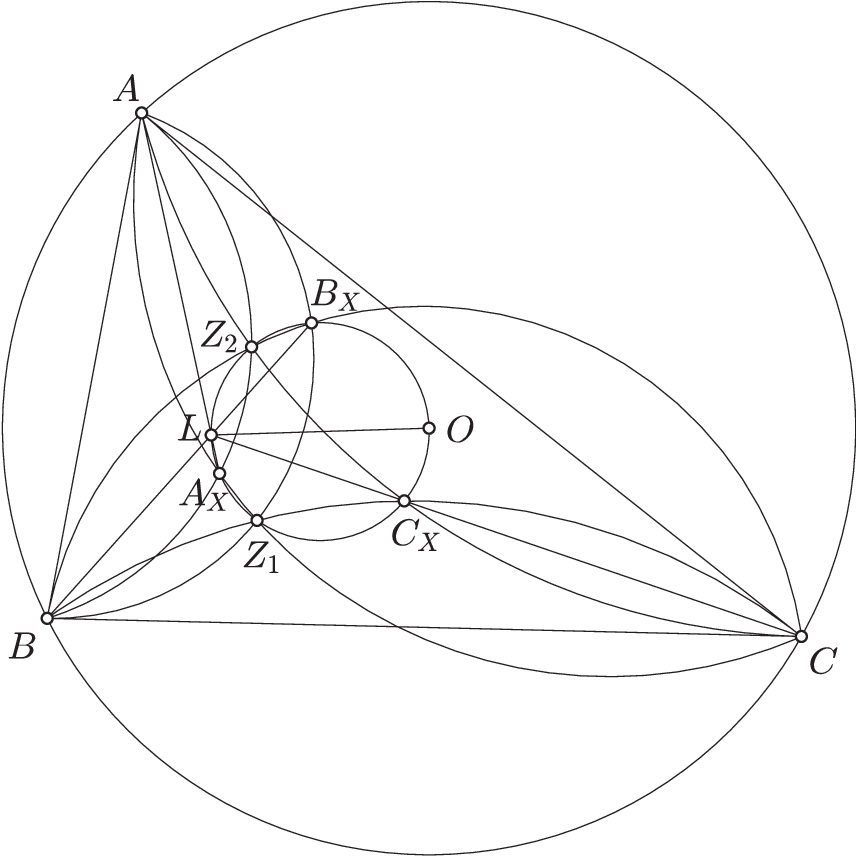}
\caption{Proof of Theorem \ref{theorem_2nd_brocard}}
\end{figure}
\end{center}

\begin{proof}
Applying Lemma \ref{lemma_mannheim}, to triangle $ABC$ with $AA_X, BB_X, CC_X$ are concurrent at $L$ and $Z_1$ is intersection of $(ABB_X), (BCC_X), (CAA_X)$ we have $L, Z_1$ lie on $(A_XB_XC_X)$. Similarly, $Z_2$ lies on $(A_XB_XC_X)$. From Lemma \ref{lemma_XYZ}, we have $A_X, B_X, C_X$ lie on $OL$. The theorem has been proved.
\end{proof}

\begin{theorem}
\label{theorem_AAY}
Ray $AA_Y$ cuts $(ABC)$ at $A_M$. Prove that $BA_YCA_M$ is parallelogram.
\end{theorem}

\begin{center}
\begin{figure}[htbp]
\includegraphics[scale=.5]{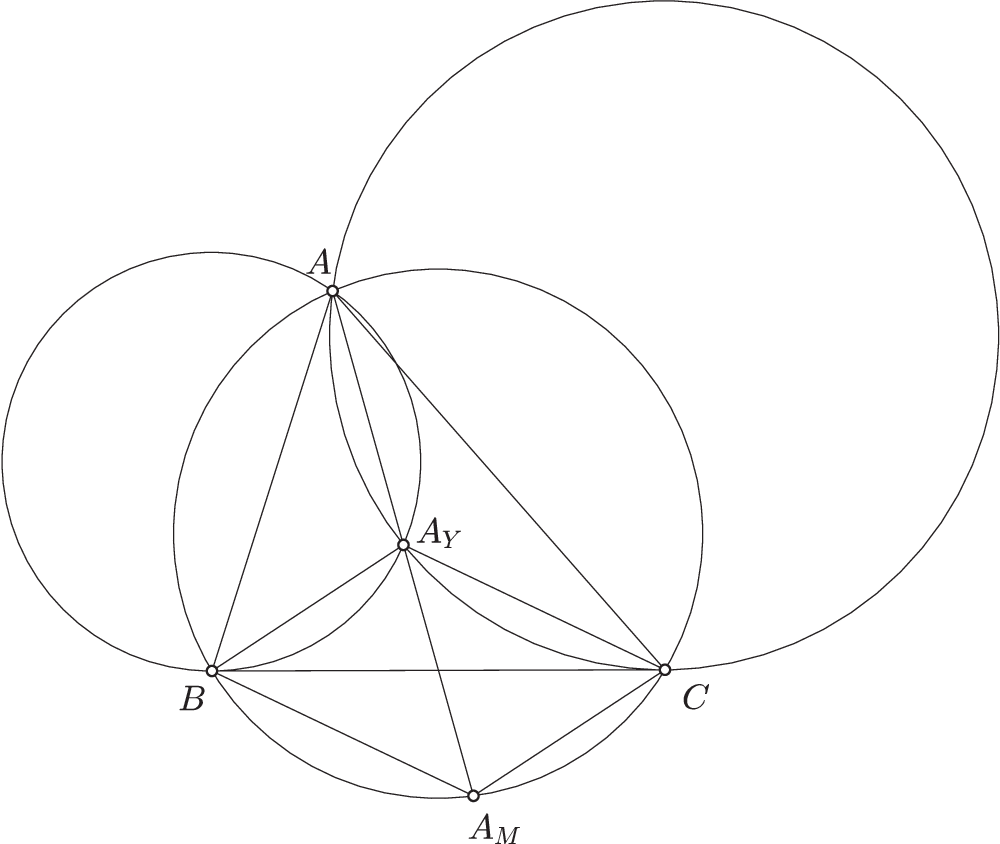}
\caption{Proof of Theorem \ref{theorem_AAY}}
\end{figure}
\end{center}

\begin{proof}
Using the properties of $A_Y$ at Problem \ref{problem_elmo}, we have
$$(BA_Y, BC)=(AB, AA_M)=(CB, CA_M).$$
Therefore $A_YB \parallel A_MC$. Similarly, $A_YC \parallel A_MC$ so $BA_YCA_M$ is a parallelogram. The problem has been solved.
\end{proof}

\begin{cor}
\label{cor_AY_sym_AS}
$A_Y$ is symmetric with $A_S$ through $BC$.
\end{cor}

\begin{theorem}
\label{theorem_AX_isogonalconjugate}
$A_X$ is the isogonal conjugate of $A_Y$ with respect to triangle $ABC$.
\end{theorem}

\begin{center}
\begin{figure}[htbp]
\includegraphics[scale=.5]{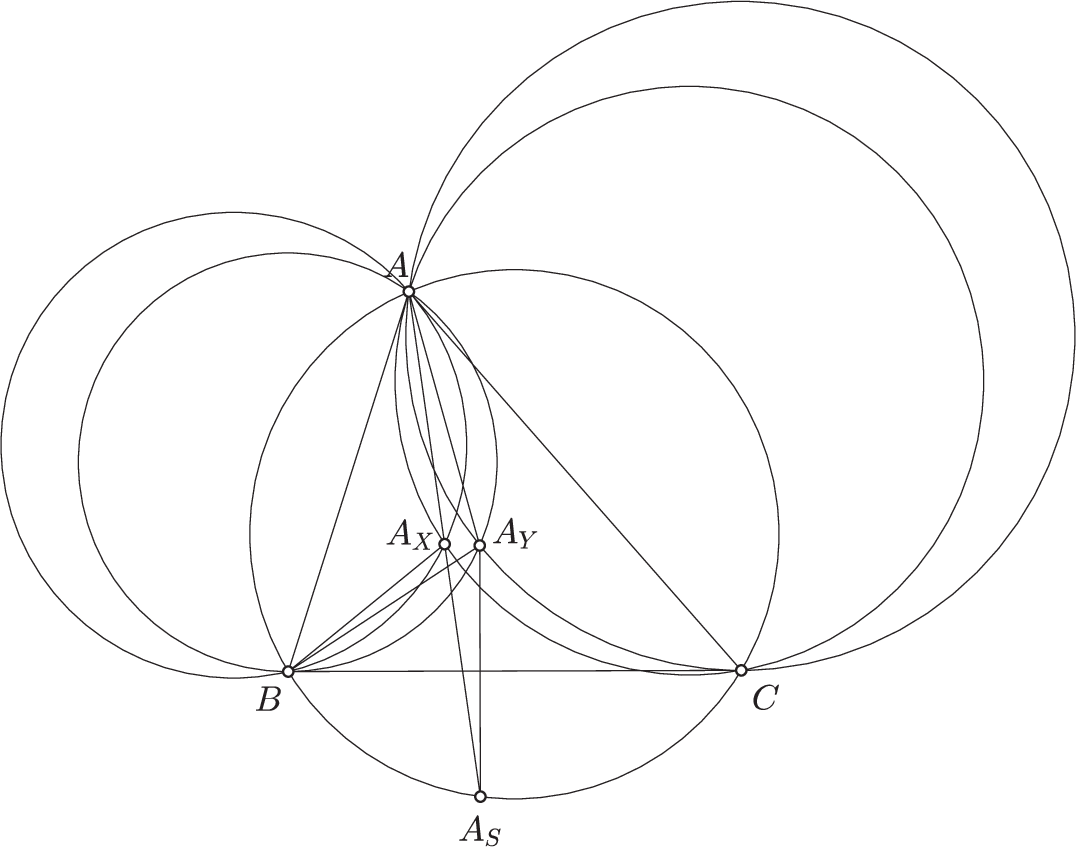}
\caption{Proof of Theorem \ref{theorem_AX_isogonalconjugate}}
\end{figure}
\end{center}

\begin{proof}
Because $AA_X$ and $AA_Y$ respectively be symmedian line and median line at vertex $A$ are two isogonal conjugate lines so we have to prove that $BA_X$ and $BA_Y$ are two isogonal conjugate ones. Through angle chasing,
$$(BA_X, BA)=(AA_X, AC)=(AB, AA_Y)=(BC, BA_Y),$$
we have $BA_X$ and $BA_Y$ are isogonal conjugate, therefore $A_X$ is the isogonal conjugate of $A_Y$ with respect to triangle $ABC$.
\end{proof}

\begin{theorem}
\label{theorem_AG_lie}
Given triangle $ABC$, centroid $G$. Let $M_{AB}$, $M_{BC}$, $M_{CA}$ respectively be midpoints of $BC$, $CA$, $AB$. Let $A_G$ be midpoint of $AA_Y$. Prove that $A_G$ lies on $(M_{AB}M_{BC}M_{CA})$.
\end{theorem}

\begin{center}
\begin{figure}[htbp]
\includegraphics[scale=.5]{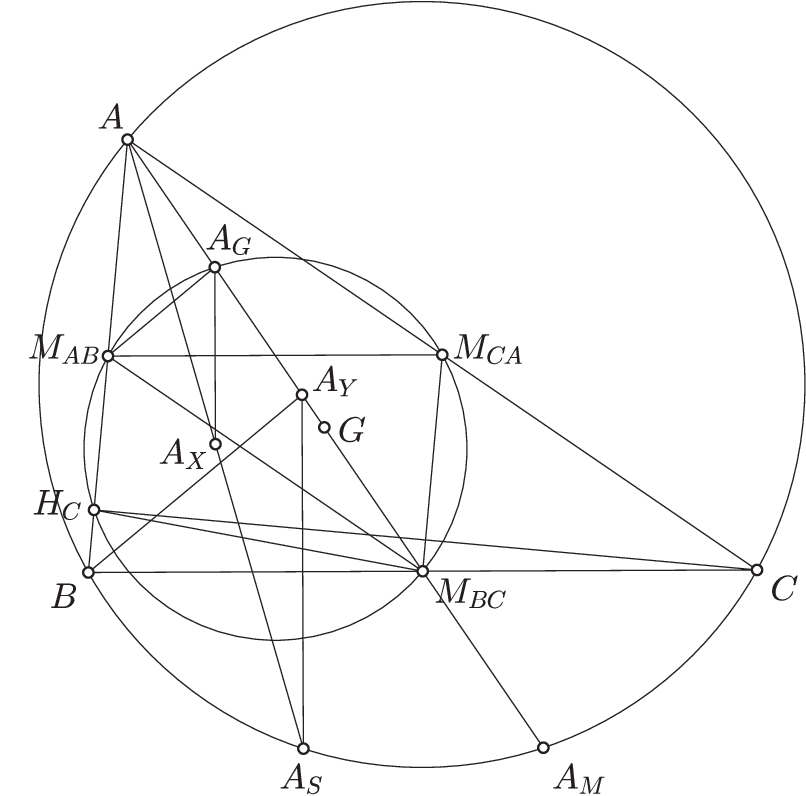}
\caption{Proof of Theorem \ref{theorem_AG_lie}}
\end{figure}
\end{center}

\begin{proof}
Because $M_{AB}A_G$ is the midline of triangle $ABA_Y$ so we have:
$$(M_{AB}A, M_{AB}A_G)=(BA, BA_Y)$$ 
Draw altitude $CH_C$ of triangle $ABC$. Through angle chasing,
$$\begin{aligned} (M_{BC}H_C, M_{BC}A) & =(M_{BC}H_C, BA)+(BA, CB)\\
& +(CB, CA)+(CA, M_{BC}A) \\
& =2(BA, CB)+(CB, CA)+(CA, M_{BC}A) \\
& =(BA, CB)+(AB, AM_{BC}) \\
& =(BA, BA_Y)+(BA_Y, CB)+(CB, CA_M) \\
& =(BA, BA_Y) \\
& =(M_{AB}A, M_{AB}A_G). \end{aligned}$$
Hence $A_G$ lies on $(M_{AB}H_CM_{BC}) \equiv (M_{AB}M_{BC}M_{CA})$. We have done.
\end{proof}

\begin{theorem}
\label{theorem_AGY_midpoint}
Draw parallelogram $M_{AB}A_GM_{CA}A_{GY}$. Then $A_{GY}$ is midpoint of $AA_M$.
\end{theorem}

\begin{center}
\begin{figure}[htbp]
\includegraphics[scale=.5]{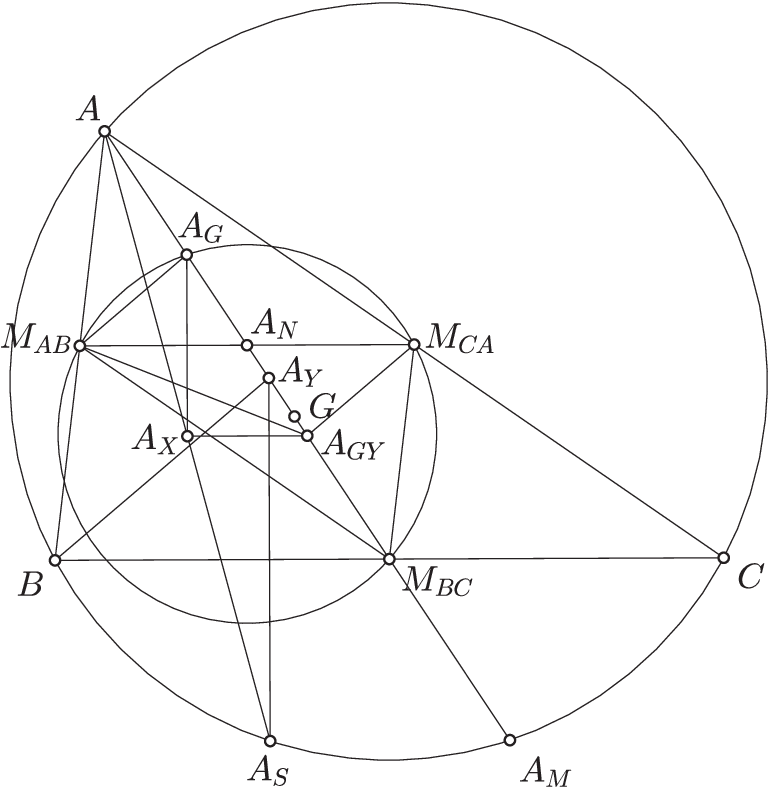}
\caption{Proof of Theorem \ref{theorem_AGY_midpoint}}
\end{figure}
\end{center}

\begin{proof}
Let $A_N$ be midpoint of $M_{AB}M_{CA}$. Because $M_{AB}A_GM_{CA}A_{GY}$ and $AM_{AB}M_{BC}M_{CA}$ are parallelogram, we can easily see that $A_N$ is midpoint of both $A_GA_{GY}$ and $AM_{BC}$. We have
$$\begin{aligned} \overline{AA_{GY}} & =\overline{AA_N}+\overline{A_NA_{GY}}=\overline{A_NM_{BC}}+\overline{A_GA_N} \\ 
& =\overline{A_GM_{BC}}=\overline{A_GA_Y}+\overline{A_YM_{BC}} \\ 
& =\frac{1}{2}\left(\overline{AA_Y}+\overline{A_YA_M}\right)=\frac{1}{2}\overline{AA_M}. \end{aligned}$$
Therefore $A_{GY}$ is midpoint of $AA_M$.
\end{proof}

\begin{cor}
\label{cor_AXAGY_parallel_BC}
$A_XA_{GY} \parallel BC$.
\end{cor}

\begin{cor}
\label{cor_AX_sym_AG}
$A_X$ is symmetric with $A_G$ through $M_{AB}M_{CA}$.
\end{cor}

\begin{theorem}
$A_{Y}$ is the Anticomplement point of $A_{GY}$ with respect to triangle $ABC$.
\end{theorem}

\begin{proof}
From these above properties, we have:
$$\begin{aligned} \overline{GA_{GY}} & =\overline{GM_{BC}}-\overline{A_{GY}M_{BC}}=2\overline{A_NG}-\overline{A_GA_Y} \\ & =\overline{2A_NG}-\overline{A_GA_N}-\overline{A_NA_Y}=\overline{A_NA_Y}+2\overline{A_YG}-\overline{A_NA_{GY}} \\ & =\overline{A_{GY}A_Y}+2\overline{A_YG}=\overline{A_{GY}G}+\overline{A_YG}. \end{aligned}$$
Hence $\overline{A_{Y}G}=2\overline{GA_{GY}}$ therefore $A_{Y}$ is the Anticomplement point of $A_{GY}$ with respect to triangle $ABC$.
\end{proof}

\begin{theorem}
\label{theorem_ZAZBZC_inscribed}
The first Brocard triangle is a triangle which has three vertices respectively be intersection points of lines $AZ_1$, $BZ_1$, $CZ_1$, $AZ_2$, $BZ_2$, $CZ_2$. Let $Z_A, Z_B, Z_C$ be three vertices of this triangle. Prove that $Z_AZ_BZ_C$ is inscribed in $(A_XB_XC_X)$.
\end{theorem}

\begin{center}
\begin{figure}[htbp]
\includegraphics[scale=.5]{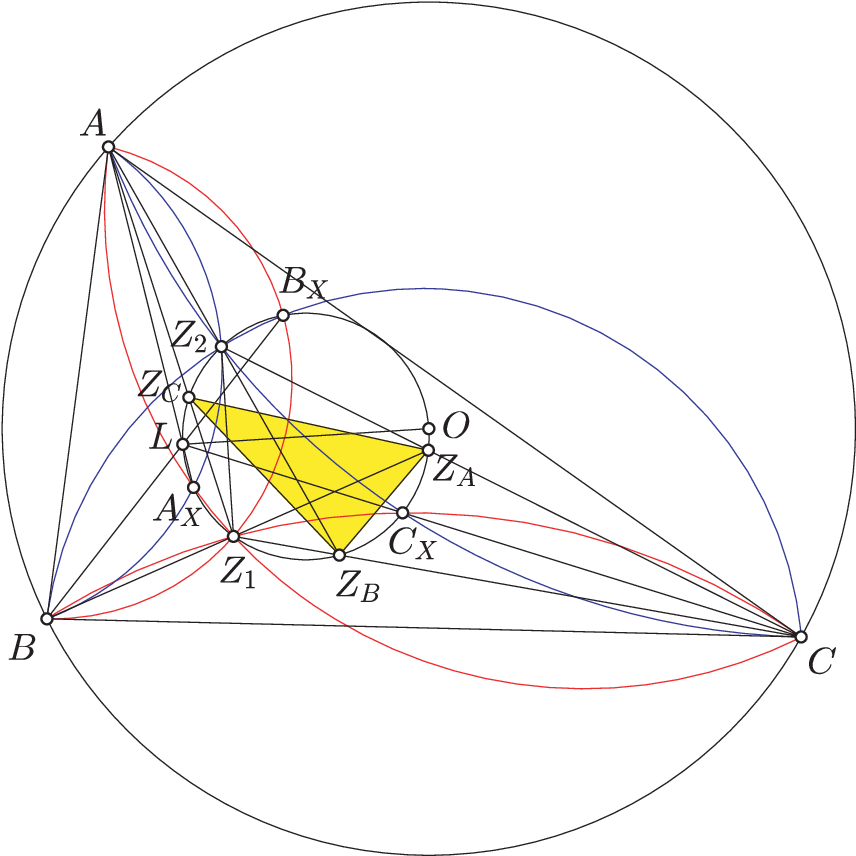}
\caption{Proof of Theorem \ref{theorem_ZAZBZC_inscribed}}
\end{figure}
\end{center}

\begin{proof}
Let $(AZ_1, AB)=(AC, AZ_2)=\alpha$. Through angle chasing,
$$\begin{aligned} (A_XZ_1, A_XZ_2) & =(A_XZ_1, A_XC)+(A_XC, A_XA)+(A_XA, A_XZ_2) \\
& = (AZ_1, AC)+(A_XA, A_XB)+(A_XA, Z_2A)+(Z_2A, A_XZ_2) \\
& = (AZ_1, AC)+(A_XA, A_XB)+(A_XA, Z_2A)+(BA, BA_X) \\
& = (AZ_1, AC)+(A_XA, A_XB)+(A_XA, Z_2A)+(AC, AA_X) \\
& = (AZ_1, AB)+(AB, AA_X)+(A_XA, A_XB)+(A_XA, Z_2A) \\
& = (AZ_1, AB)+(BA, A_XB)+(A_XA, Z_2A) \\
& = (AZ_1, AB)+(AC, AA_X)+(AA_X, AZ_2) \\
& = (AZ_1, AB)+(AC, AZ_2) \\
& = 2\alpha \text{ (mod } \pi). \end{aligned}$$
On the other hand, we have
$$(Z_AZ_1, Z_AZ_2)=(Z_AZ_1, CB)+(CB, Z_AZ_2)=\alpha+\alpha=2\alpha.$$
Therefore $Z_A$ lies on $(Z_1A_XZ_2) \equiv (A_XB_XC_X)$. Similarly, $Z_B, Z_C$ lie on $(A_XB_XC_X)$. WE have done.
\end{proof}

\begin{cor}
$Z_1$ is symmetric with $Z_2$ through $OL$ and $(OZ_1, OZ_2)=2\alpha$.
\end{cor}

\begin{theorem}
\label{theorem_OWCWparallelOZA}
$OA_X$, $CZ_2$ cut $\omega_1$ at $O_W$, $C_W$. Then $O_WC_W \parallel OZ_A$.
\end{theorem}

\begin{lemma}
\label{lemma_BWCWparallelBC}
Let $BZ_1$ cut $\omega_2$ at $B_W$. Prove that $B_WC_W \parallel BC$.
\end{lemma}

\begin{proof}
We have $(BB_W, BC)=(CA, CZ_1)=(B_WA, B_WZ_1)$ so $AB_W \parallel BC$. Similarly, $AC_W \parallel BC$ so $B_WC_W \parallel BC$. We have done.
\end{proof}

Back to Theorem \ref{theorem_OWCWparallelOZA},

\begin{center}
\begin{figure}[htbp]
\includegraphics[scale=.5]{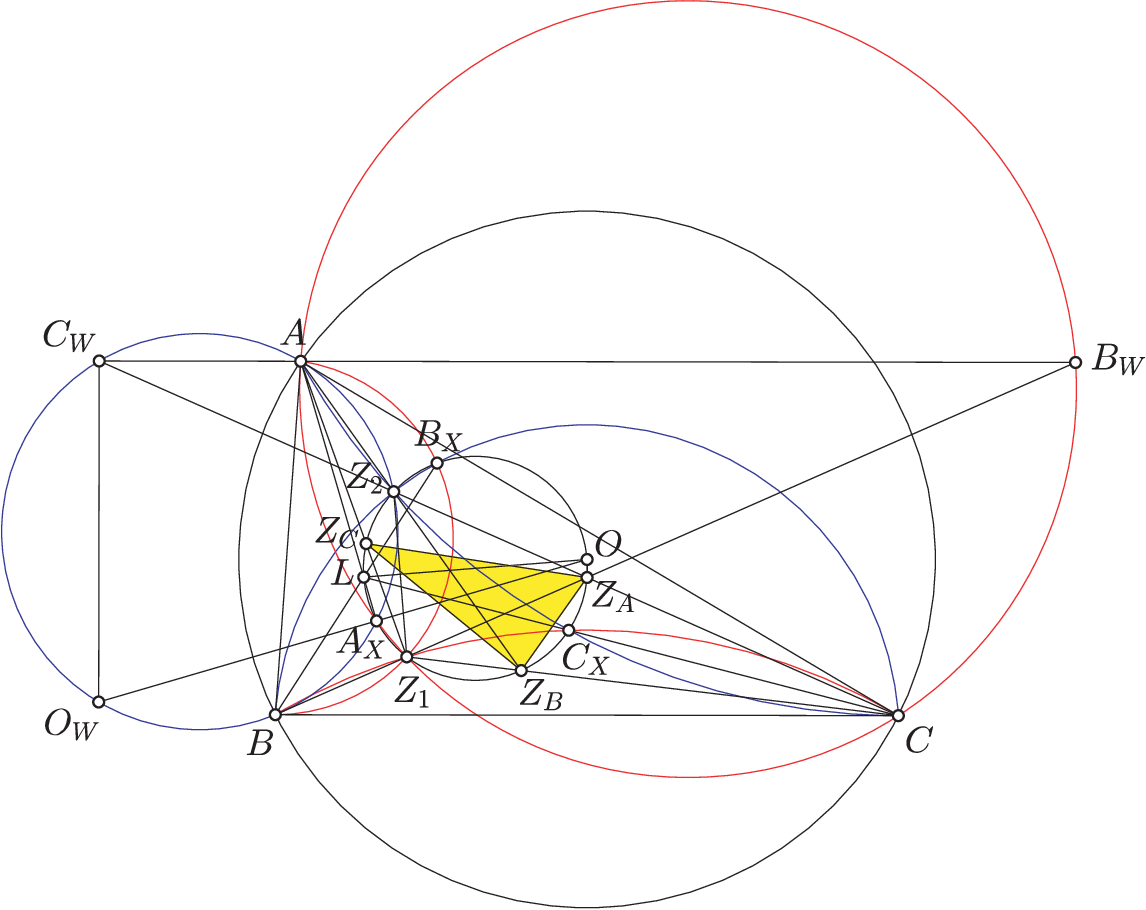}
\caption{Proof of Theorem \ref{theorem_OWCWparallelOZA}}
\end{figure}
\end{center}

\begin{proof}
By Lemma \ref{lemma_BWCWparallelBC} and $(BC, BZ_A)=(CZ_A, CB)$, we have $BCB_WC_W$ is an isosceles trapezoid or $OZ_A$ is perpendicular to $B_WC_W$. On the other hand, by Theorem \ref{theorem_2nd_brocard}, we have $(A_XA, A_XO_W)=(C_WA, C_WO_W)=\frac{\pi}{2}$. Therefore $O_WC_W$ is perpendicular to $B_WC_W$. So $O_WC_W \parallel OZ_A$.
\end{proof}

\begin{theorem}
\label{theorem_concurrentG}
$A_XZ_A$, $B_XZ_B$, $C_XZ_C$ are concurrent at $G$.
\end{theorem}

\begin{lemma}
\label{lemma_parallelogramACDB}
Given triangle $ABC$. On the plane containing this triangle, choose $A'$, $B'$, $C'$ such that the triangles $A'CB$, $B'AC$, $C'BA$ are homothetic. We construct the parallelogram $BA'CD$. Prove that $A'CDB$ is also a parallelogram.
\end{lemma}

\begin{center}
\begin{figure}[htbp]
\includegraphics[scale=.5]{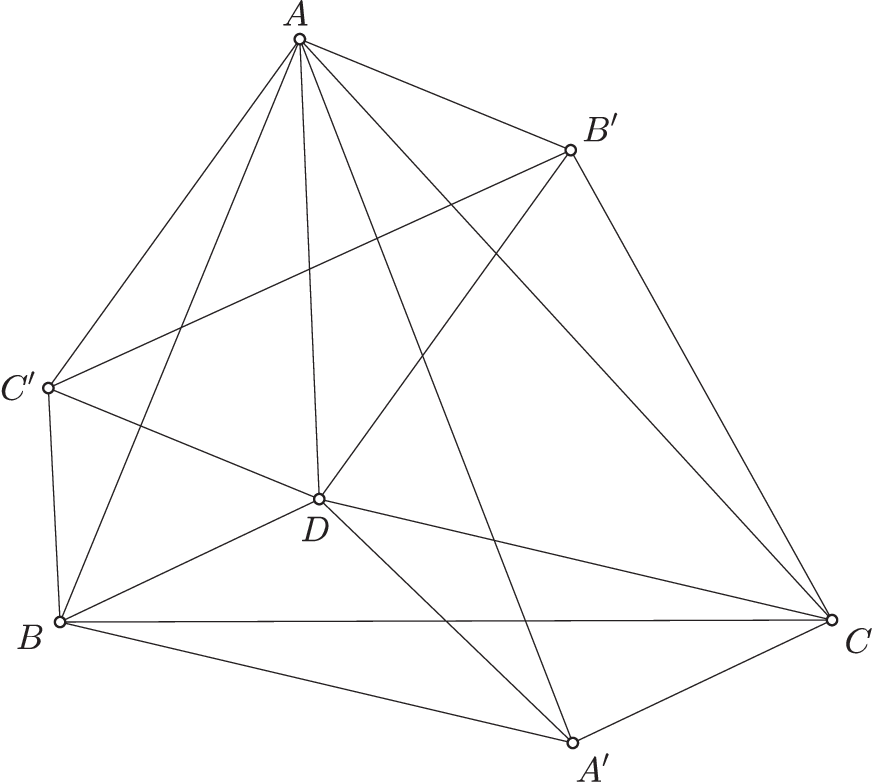}
\caption{Proof of Lemma \ref{lemma_parallelogramACDB}}
\end{figure}
\end{center}

\begin{proof}
We have
$$\frac{C'B}{C'D}=\frac{C'B}{AB'}=\frac{C'A}{B'C}=\frac{B'D}{B'C}$$
and
$$\begin{aligned}(C'B, C'D) &=(C'B, C'A)+(C'A, C'D)\\
&=(B'A, B'C)+(B'D, B'A)=(B'D, B'C).
\end{aligned}$$
Therefore triangle $C'BD$ and triangle $B'DC$ are homothetic. Hence
$$\frac{BD}{CD}=\frac{C'B}{B'D}=\frac{C'B}{C'A}=\frac{A'B}{A'C}.$$
On the other hand,
$$\begin{aligned} (DB, DC) & =(DB, DC')+(DC', DB')+(DB', DC) \\ & =(DC, B'C)+(B'A, DB')+(DB', DC) \\ & =(B'A, B'C) \\ & =(A'C, AB). \end{aligned}$$
So triangle $DBC$ and triangle $A'CB$ are homothetic. Therefore, $BD \parallel A'C$ and $CD \parallel A'B$. We conclude that $A'CDB$ is a parallelogram.
\end{proof}

\begin{lemma}
\label{lemma_centroid_ZAZBZC}
$G$ is the centroid of triangle $Z_AZ_BZ_C$.
\end{lemma}

\begin{center}
\begin{figure}[htbp]
\includegraphics[scale=.5]{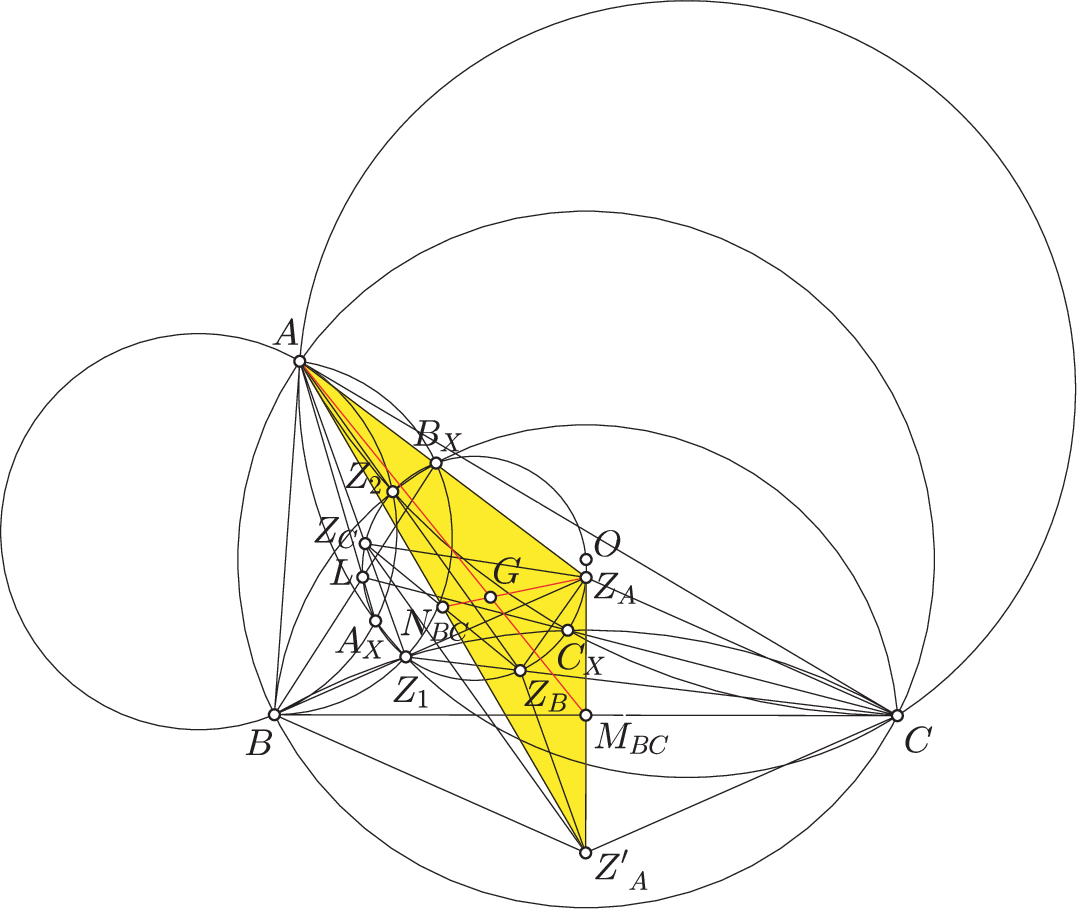}
\caption{Proof of Lemma \ref{lemma_centroid_ZAZBZC}}
\end{figure}
\end{center}

\begin{proof}
Let $Z'_A$ be symmetric point with $Z_A$ through $BC$. Since triangle $BZ_AC$ is a isosceles triangle at $Z_A$, $BZ_ACZ'_A$ is a parallelogram. Applying Lemma \ref{lemma_parallelogramACDB} for triangle $ABC$, we have $AZ_CZ'_AZ_B$ is a parallelogram so $AZ'_A$ passes midpoint $N_{BC}$ of $Z_BZ_C$. In triangle $AZ_AZ'_A$, the two median lines $AM_{BC}$ and $Z_AN_{BC}$ intersect with each other at this triangle's centroid so that this point is also the centroid $G$ of triangle $ABC$ and triangle $Z_AZ_BZ_C$.
\end{proof}

\begin{lemma}
\label{lemma_conccurentXXYYZZ}
Given triangle $ABC$ and two disjoint arbitrary points $P$, $Q$. Let $X$, $Y$, $Z$ respectively be projection of $P$ on the lines which passes $Q$ and is perpendicular to $BC$, $CA$, $AB$. Let $X'$, $Y'$, $Z'$ respectively be projection of $Q$ on $AP$, $BP$, $CP$. Prove that $XX'$, $YY'$ and $ZZ'$ are concurrent.
\end{lemma}

\begin{center}
\begin{figure}[htbp]
\includegraphics[scale=.4]{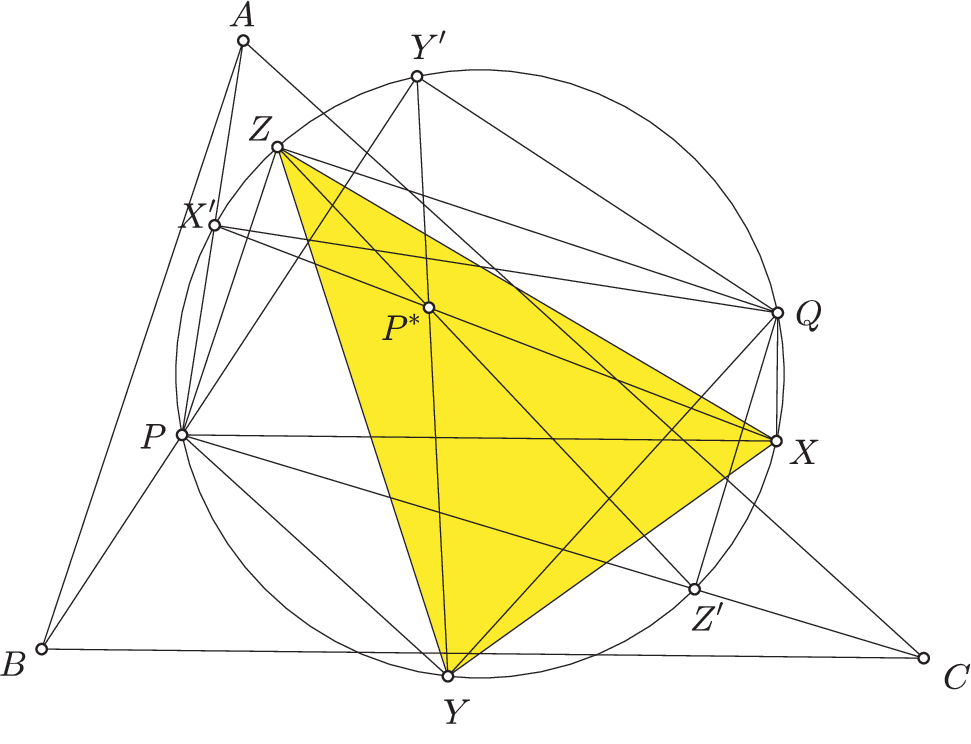}
\caption{Proof of Lemma \ref{lemma_conccurentXXYYZZ}}
\end{figure}
\end{center}

\begin{proof}
Noted that $X$, $Y$, $Z$, $X'$, $Y'$, $Z'$ lie on a circle with diameter $PQ$ and $AB \parallel PZ$, $AC \parallel PY$ so we have
$$(AB, AC)=(ZP, PY)=(XY, XZ).$$
Similarly, $(BC, BA)=(YZ, YX)$ and $(CA, CB)=(ZX, ZY)$ so two triangles $XYZ$ and $ABC$ are homothetic. Let $P'$ be the image of $P$ through the homothety transformation which turns triangle $ABC$ into triangle $XYZ$. Let $P^*$ be the isogonal conjugate point of $P$ in triangle $XYZ$. We have
$$(XX', XZ)=(PX', PZ)=(AP, AB)=(XP^*, XZ).$$
Therefore $XX'$ passes $P^*$. Similarly, $YY'$ and $ZZ'$ passes $P^*$, which solves the lemma.
\end{proof}

\begin{center}
\begin{figure}[htbp]
\includegraphics[scale=.41]{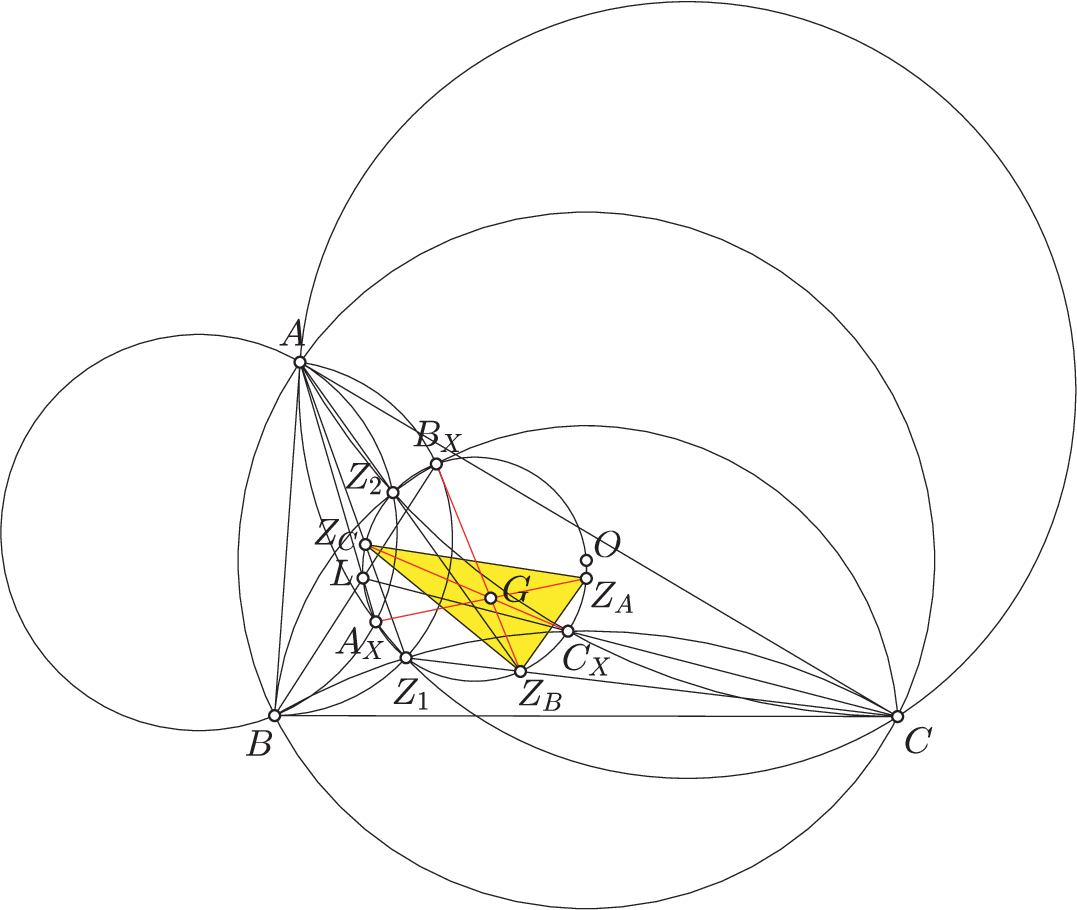}
\caption{Proof of Theorem \ref{theorem_concurrentG}}
\end{figure}
\end{center}

Back to Theorem \ref{theorem_concurrentG},

\begin{proof}
By applying Lemmas \ref{lemma_centroid_ZAZBZC} and \ref{lemma_conccurentXXYYZZ} for triangle $Z_AZ_BZ_C$, we solves the problem.
\end{proof}

By Corollary \ref{cor_AX_sym_AG}, the Brocard circle is definitely the Hagge circle \cite{hagge} of $G$ in triangle $M_{BC}M_{CA}M_{AB}$. We will discuss some properties related to this issue and use these ones to prove Theorem \ref{theorem_concurrentG}.

\begin{lemma}[\bf Hagge circle]
\label{lemma_hagge}
Given triangle $ABC$, orthocenter $H$ and a point $P$ lies inside $(ABC)$. Rays $AP$, $BP$, $CP$ cuts $(ABC)$ at the second points $A_1$, $B_1$, $C_1$. Let $A_2$, $B_2$, $C_2$ respectively be the symmetry points of $A_1$, $B_1$, $C_1$ through $A$, $B$, $C$. Prove that $A_2$, $B_2$, $C_2$ and $H$ are cyclic.
\end{lemma}

\begin{rem}
The circle mentioned above called Hagge circle $\mathcal{H}_P$ of point $P$ with respect to triangle $ABC$.
\end{rem}

\begin{center}
\begin{figure}[htbp]
\includegraphics[scale=.5]{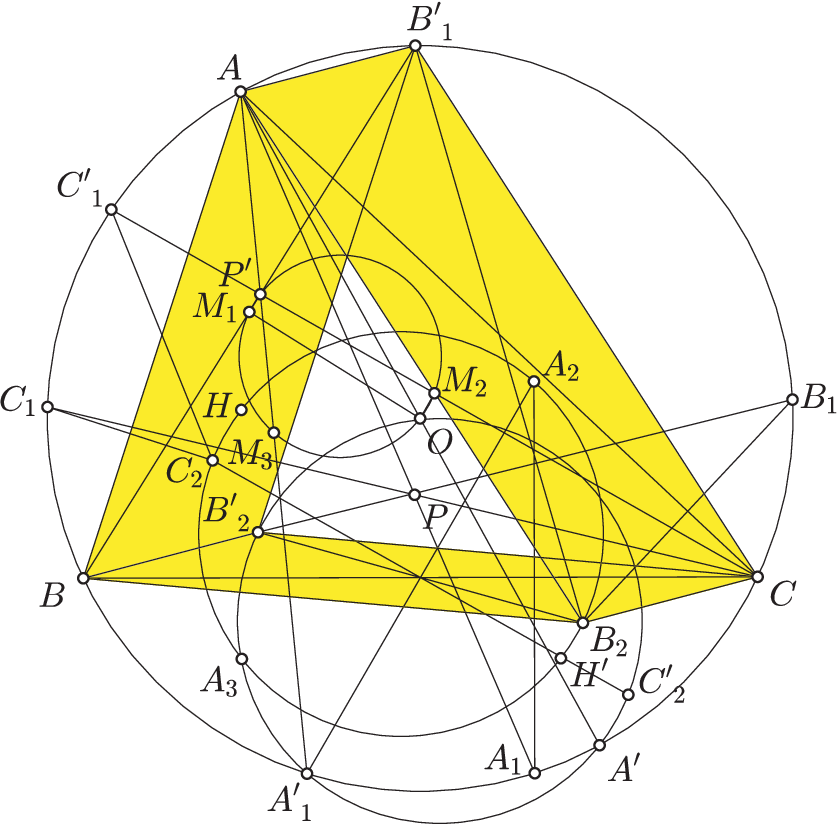}
\caption{Proof of Lemma \ref{lemma_hagge}}
\end{figure}
\end{center}

\begin{proof}
Let $P'$ be the isogonal conjugate of $P$ in triangle $ABC$. $AP'$, $BP'$, $CP'$ cut $(ABC)$ respectively at $A'_1$, $B'_1$, $C'_1$. Let $B'_2$, $C'_2$ respectively be symmetry points of $B_2$, $C_2$ through midpoint of $BC$. Set $O$ is the center of $(ABC)$ with diameter $AA'$. It is easy to see that $A_2$, $B_2$, $C_2$ are symmetry points of $A'_1$, $B'_2$, $C'_2$ through midpoint of $BC$. We have $CB_2AB'_1$ and $CB_2BB'_2$ are parallelograms so the quadrilateral $ABB'_2B'_1$ is also a parallelogram. We conclude that midpoint $M_1$ of $AB'_2$ is the projection of $O$ on $BP'$. Similarly, midpoint $M_2$ of $AC'_2$ is the projection of $O$ on $CP'$. Let $M_3$ be midpoint of $AA'_1$. Then triangle $M_1M_2M_3$ is inscribed in a circle with diameter $OP'$. We use two geometric transformations respectively with the first one (1) is the homothety transformation $\mathbf{H}_A^2$ with center $A$ and ratio 2; and the later (2) is the symmetric one $\mathbf{S}_{M_{BC}}$ through midpoint $M_{BC}$ of $BC$, to points $M_3$, $M_1$, $M_2$, $O$ such that
$$\mathbf{S}_{M_{BC}} \circ \mathbf{H}_{A}^{2}\colon M_3, M_1, M_2, O \mapsto A_2, B_2, C_2, H.$$
Since $M_1, M_2, M_3, O$ lie on a same circle so $A_2, B_2, C_2, H$ are also cyclic.
\end{proof}

\begin{lemma}
\label{lemma_ratio}
Given triangle $ABC$ and two points $S$ and $S'$ are isogonal conjugate. Let $X$, $X'$ respectively be intersection of $AS$, $AS'$ with $(ABC)$. Let $V$ be intersection of $AS'$ and $BC$. Prove that
$$\frac{\overline{AS}}{\overline{SX}}=\frac{\overline{S'V}}{\overline{VX'}}.$$
\end{lemma}

\begin{center}
\begin{figure}[htbp]
\includegraphics[scale=.5]{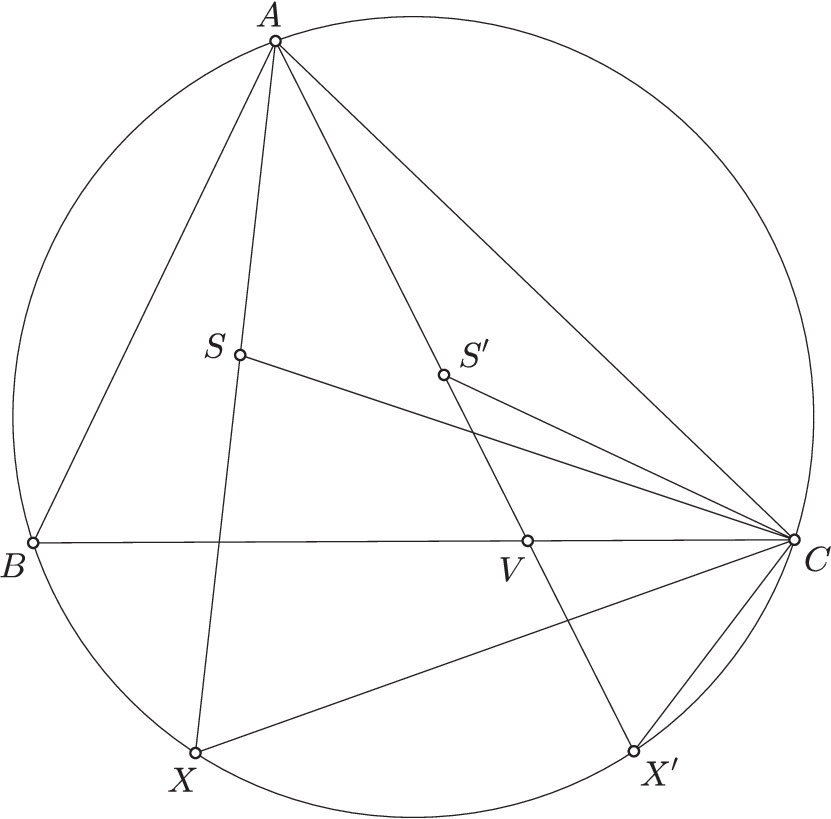}
\caption{Proof of Lemma \ref{lemma_ratio}}
\end{figure}
\end{center}

\begin{proof}
We have
$$(CV, CX')=(AB, AX')=(AX, AC),$$
$$(X'V, X'C)=(XA, XC)$$
So two triangles $VCX'$ and $CAX$ are homothetic. On the other hand,
$$\begin{aligned} (S'X', S'C) & =(AS', AC)+(CA, CS') \\ & =(AB, AS)+(CS, CB) \\ & =(CB, CX)+(CS, CB) \\ & =(CS, CX), \end{aligned}$$
$$(X'S', X'C)=(XS, XC).$$
Therefore two triangles $CX'S'$ and $SXC$ are homothetic. So, $$\overline{X'V}\cdot\overline{XA}=\overline{X'C}\cdot\overline{XC}=\overline{X'S'}\cdot\overline{XS}.$$
Hence, $\frac{\overline{XA}}{\overline{XS}}=\frac{\overline{X'S'}}{\overline{X'V}}$ or $\frac{\overline{AS}}{\overline{SX}}=\frac{\overline{S'V}}{\overline{VX'}}$.
\end{proof}

\begin{lemma}
\label{lemma_P_concurrent}
Let $A_3$, $B_3$, $C_3$ respectively be intersection of $\mathcal{H}_P$ and $AH$, $BH$, $CH$. Prove that $A_2A_3$, $B_2B_3$, $C_2C_3$ are concurrent at $P$.
\end{lemma}

\begin{center}
\begin{figure}[htbp]
\includegraphics[scale=.5]{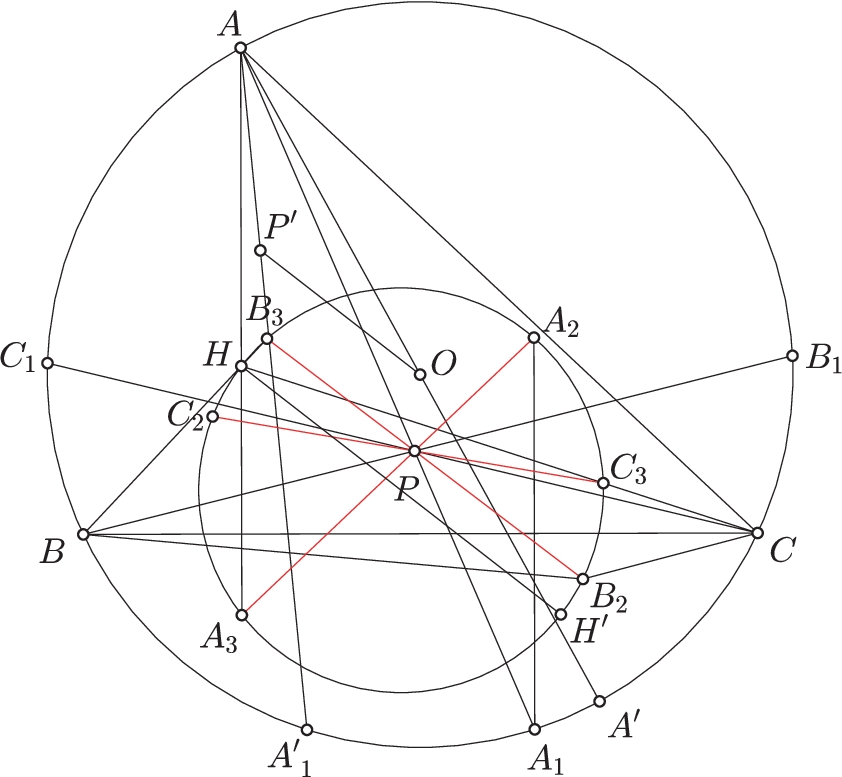}
\caption{Proof of Lemma \ref{lemma_P_concurrent}}
\end{figure}
\end{center}

\begin{proof}
We construct diameter $HH'$ of $\mathcal{H}_P$. From Lemma \ref{lemma_hagge}, we have
$$\mathbf{S}_{M_{BC}} \circ \mathbf{H}_{A}^{2}\colon O, P' \mapsto H, H'.$$
Therefore $OP' \parallel HH'$ and $\overline{H'H}=2\overline{OP'}$ so $H'$ is the Anticomplement point of $P'$ in triangle $ABC$. On the other hand, $A_3H' \parallel BC$ and $BC \perp AH$ hence $AA'=2\text{d}(P', BC)$. Using Lemma \ref{lemma_ratio}, we have
$$\frac{\overline{AP}}{\overline{PA_1}}=\frac{\text{d}(P', BC)}{\text{d}(A'_1, BC)}=\frac{\text{d}(P', BC)}{\text{d}(A_1, BC)}=\frac{\overline{AA_3}}{\overline{A_2A_1}}.$$
This means $A_2A_3$ passes $P$. Similarly, $B_2B_3, C_2C_3$ passes $P$. We have solved the lemma.
\end{proof}

Back to Theorem \ref{theorem_concurrentG},

\begin{center}
\begin{figure}[htbp]
\includegraphics[scale=.45]{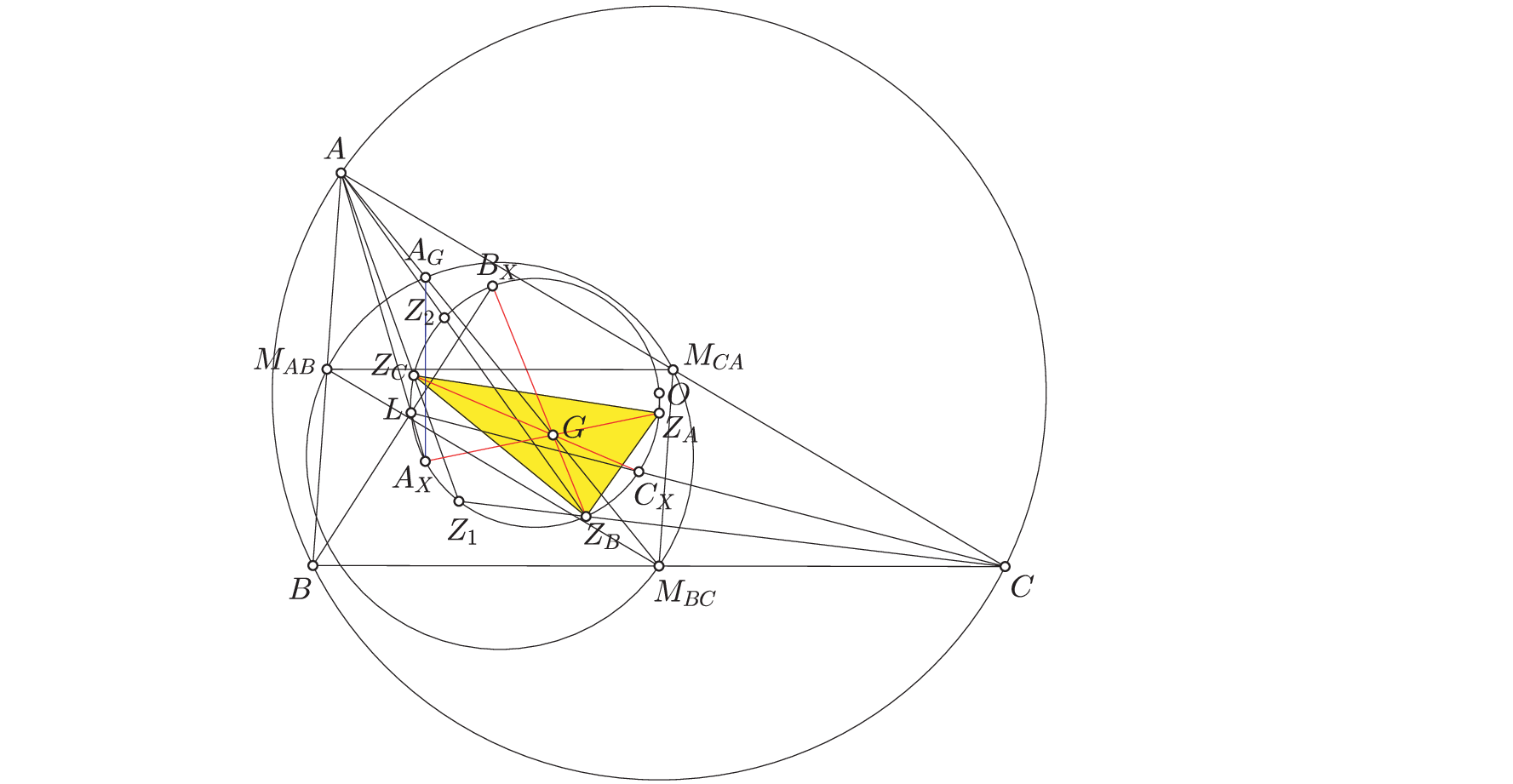}
\caption{Proof of Theorem \ref{theorem_concurrentG}}
\end{figure}
\end{center}

\begin{proof} By Corollary \ref{cor_AX_sym_AG}, the Brocard cirlce is the Hagge circle of $G$ with respect to triangle $M_{BC}M_{CA}M_{AB}$. By applying Lemma \ref{lemma_hagge} and \ref{lemma_P_concurrent} for triangle $M_{BC}M_{CA}M_{AB}$, we have the result. \end{proof}

\section{Properties related to the second problem}

\begin{theorem}
\label{theorem_AY_median}
$A_Y$ is the projection of orthocenter $H$ on the $A-$median line of triangle $ABC$.
\end{theorem}

\begin{center}
\begin{figure}[htbp]
\includegraphics[scale=.5]{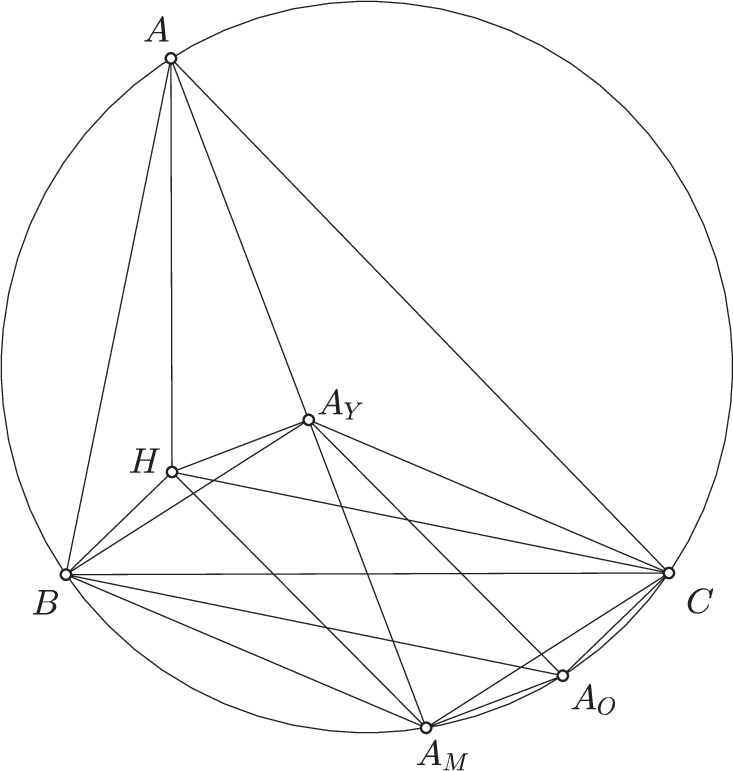}
\caption{Proof of Theorem \ref{theorem_AY_median}}
\end{figure}
\end{center}

\begin{proof}
Let $AA_O$ be diameter of $(ABC)$. Since $BHCA_O$ is a parallelogram, $M_{BC}$ is the midpoint of $A_OH$. On the other hand, $M_{BC}$ is also the mid point of $A_YA_M$. By Theorem \ref{theorem_AAY}, $HA_YA_OA_M$ is a parallelogram therefore $A_YH \parallel A_OA_M$. Moreover, $A_OA_M \perp AA_M$ hence $A_Y$ is the projection of $H$ on the $A$-median line.
\end{proof}

\begin{cor}
Let $B_Y$ and $C_Y$ be constructed in a similar manner as $A_Y$, but with respect to $B$ and $C$ respectively. Let $H_AH_BH_C$ be the Orthocentroidal triangle with $H_A$, $H_B$, $H_C$ repsectively be projection of $G$ on $AH$, $BH$, $CH$. Then two triangles $A_YB_YC_Y$ and $H_AH_BH_C$ are inscribed in the circle with diameter $GH$ which called Orthocentroidal circle.
\end{cor}

\begin{cor}
\label{cor_Hagge_Orthocentroidal}
The Hagge circle of the Lemoine point $L$ in triangle $ABC$ is the Orthocentroidal circle of this triangle.
\end{cor}

\begin{center}
\begin{figure}[htbp]
\includegraphics[scale=.5]{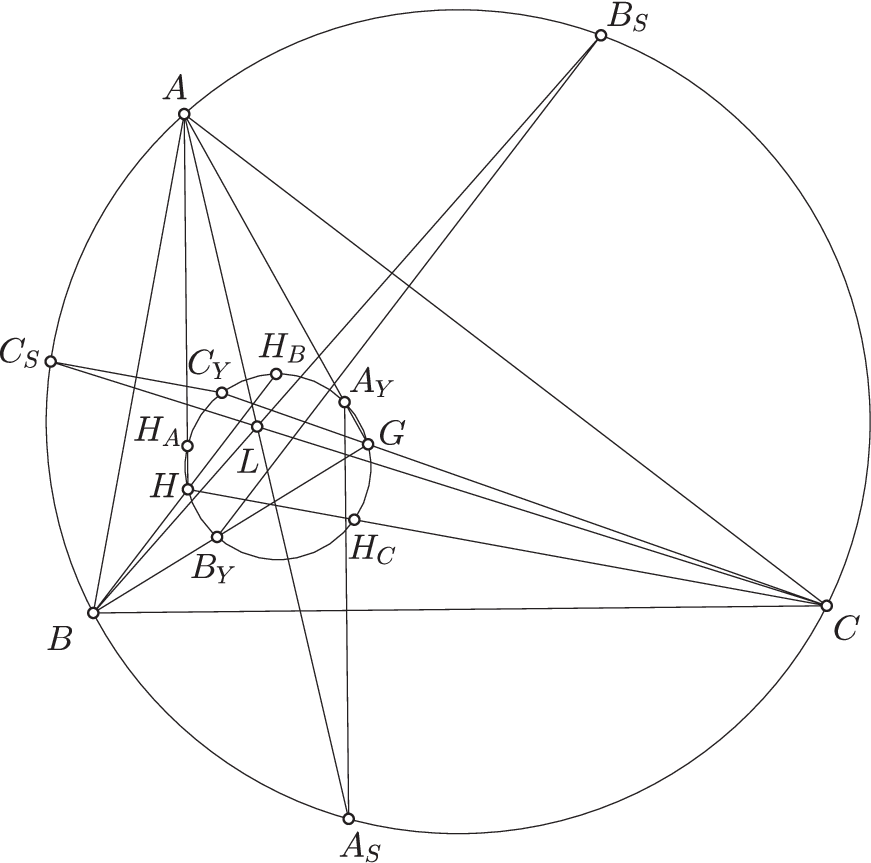}
\caption{Demonstration of Corollary \ref{cor_Hagge_Orthocentroidal}}
\end{figure}
\end{center}

\begin{theorem}
\label{theorem_X39}
The radical line of $\left(A_XB_XC_X\right)$ and $\left(A_YB_YC_Y\right)$ passes midpoint $X\left(39\right)$ of two Brocard points.
\end{theorem}

\begin{center}
\begin{figure}[htbp]
\includegraphics[scale=.5]{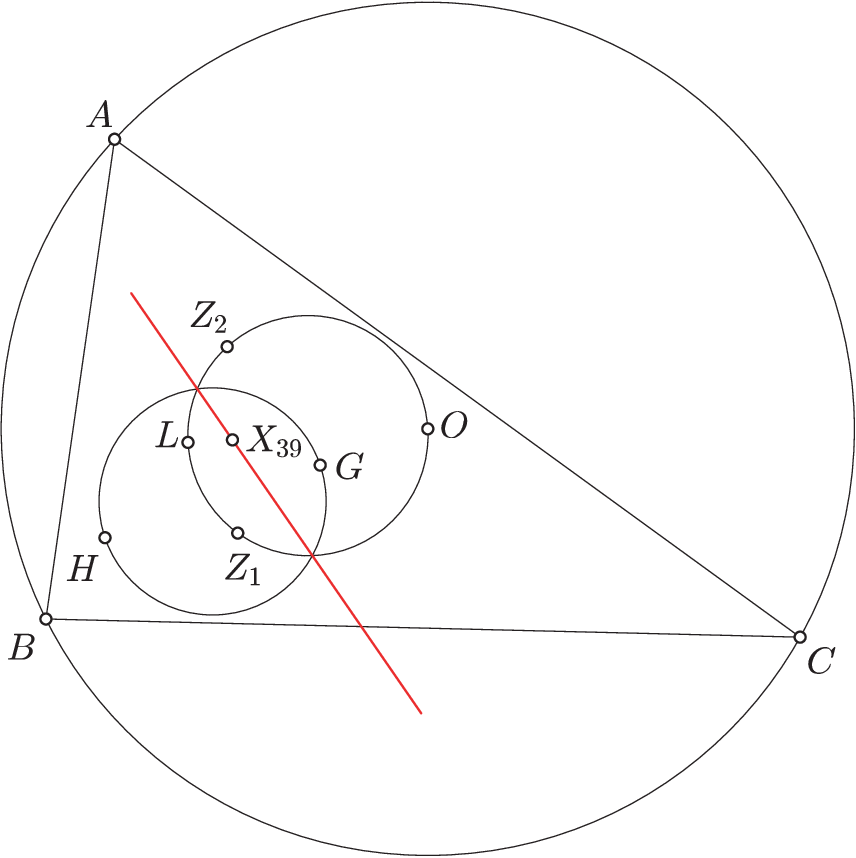}
\caption{Demonstration of Theorem \ref{theorem_X39}}
\end{figure}
\end{center}

We can see the proof of this property here \cite{brocardmidpoint} using Barycentric coordinate.

\begin{theorem}
\label{theorem_L_concurrent}
$A_YH_A$, $B_YH_B$, $C_YH_C$ are concurrent at $L$.
\end{theorem}

\begin{lemma}
\label{lemma_lemoine_sub}
$L$ is also the Lemoine point of triangle $H_AH_BH_C$.
\end{lemma}

\begin{proof}
By Corollary \ref{cor_Hagge_Orthocentroidal}, the Orthocentroidal circle of triangle $ABC$ is the Hagge circle of $L$ with respect to this triangle. Moreover, by Corollary \ref{cor_AY_sym_AS}, triangle $A_YB_YC_Y$ is the Hagge triangle of $L$, which means $A_Y$, $B_Y$, $C_Y$ are constructed in the same way of $A_2$, $B_2$, $C_2$ in Lemma \ref{lemma_hagge}, but with respect to $L$. We then apply Lemma \ref{lemma_P_concurrent} to completely solve this result.
\end{proof}

Back to Theorem \ref{theorem_L_concurrent},

\begin{proof} By applying Lemma \ref{lemma_hagge}, \ref{lemma_P_concurrent} and \ref{lemma_lemoine_sub} for triangle $ABC$ with Lemoine point $L$ and the Orthocentroidal circle, we can prove the theorem. \end{proof}

\begin{center}
\begin{figure}[htbp]
\includegraphics[scale=.5]{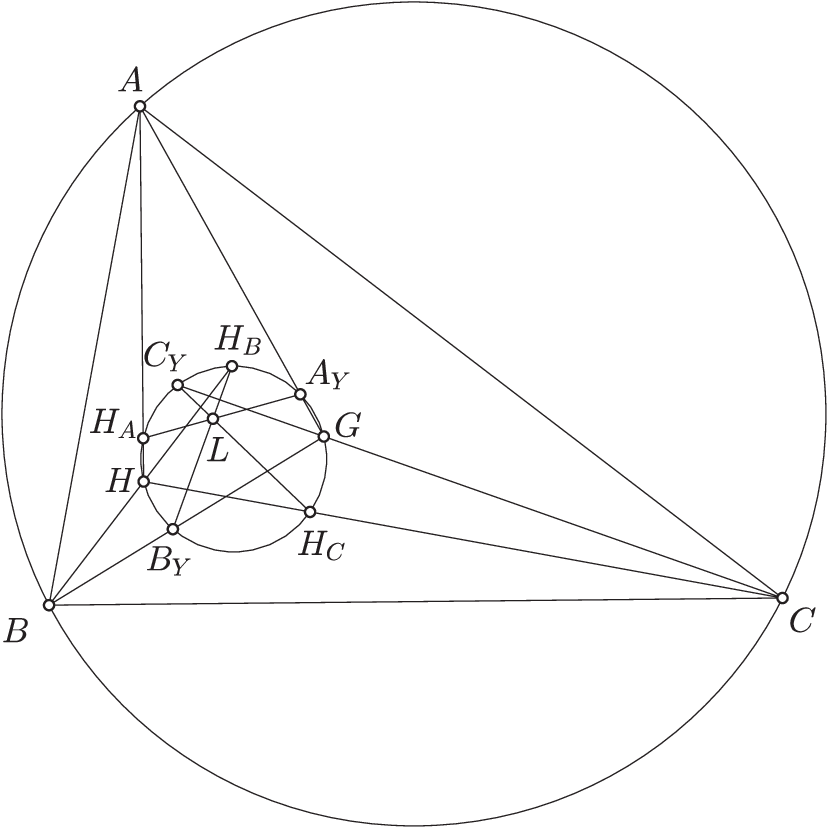}
\caption{Proof of Theorem \ref{theorem_L_concurrent}}
\end{figure}
\end{center}

\begin{theorem}
\label{theorem_AXAY_parallel_LNMBC}
The perpendicular bisector of $AXA_{GY}$ cuts $AL$ at $L_N$. Then $AXA_Y \parallel L_NM_{BC}$.
\end{theorem}

\begin{center}
\begin{figure}[htbp]
\includegraphics[scale=.5]{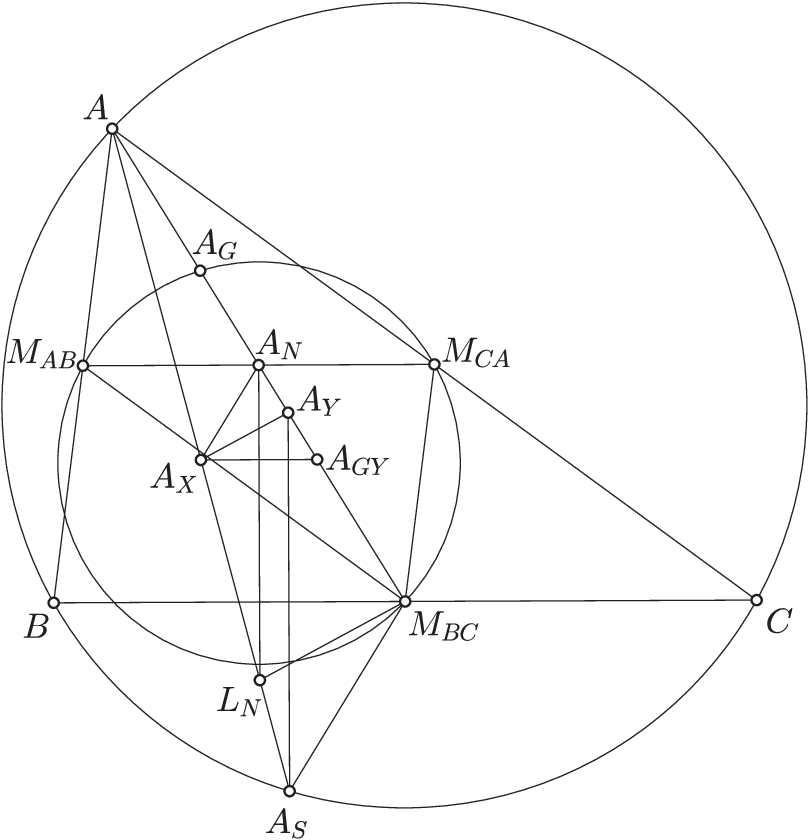}
\caption{Proof of Theorem \ref{theorem_AXAY_parallel_LNMBC}}
\end{figure}
\end{center}

\begin{proof}
By Corollary \ref{cor_AXAGY_parallel_BC}, $L_N$ also lies on the perpendicular bisector of $M_{AB}M_{CA}$ containing $A_N$ so $A_NL_N \parallel A_YA_S$. Therefore $$\overline{AL_N}\cdot\overline{AA_Y}=\overline{AA_N}\cdot\overline{AA_S}=\overline{AM_{BC}}\cdot\overline{AA_X}$$ so $\frac{\overline{AA_X}}{\overline{AL_N}}=\frac{\overline{AA_Y}}{\overline{AM_{BC}}}$. Hence, $A_XA_Y \parallel L_NM_{BC}$.
\end{proof}

\begin{theorem}
\label{theorem_AXAGYLN_tangent}
$\left(A_XA_{GY}L_N\right)$ is tangent to $\left(BOC\right)$.
\end{theorem}

\begin{center}
\begin{figure}[htbp]
\includegraphics[scale=.5]{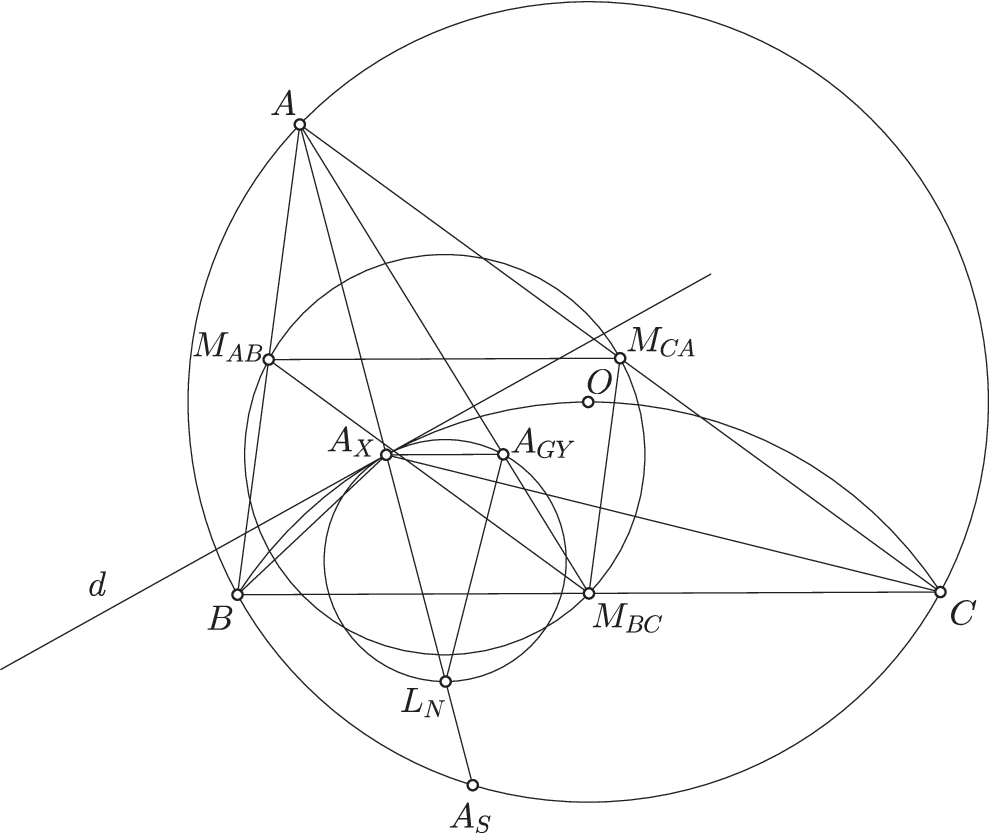}
\caption{Proof of Theorem \ref{theorem_AXAGYLN_tangent}}
\end{figure}
\end{center}

\begin{proof} By Theorem \ref{theorem_AX_project}, $A_X$ lies on $(BOC)$. Let $Ad$ be a ray which belongs to tangent line at $A$ of $\left(A_XA_{GY}L_N\right)$ such that $Ad$ and $B$ are on the same side of the plane separated by $AL$. By Theorem \ref{theorem_AX_project}, $A_XA_S$ is the $A_X$-bisector of triangle $A_XBC$. We have
$$\begin{aligned} (A_Xd, A_XB) & =(A_Xd, A_XA_S)+(A_XA_S, A_XB) \\
& =(A_XA_S, A_XA_{GY})+(A_XC, A_XA_S) \\
& =(A_XC, A_XA_{GY}) \\
& =(CB, CA_X). \end{aligned}$$
So $Ad$ is the tangent line of $(A_XBC) \equiv (BOC)$. We have done.
\end{proof}

\begin{theorem}
\label{theorem_equal}
Let $A_{LX}$ be projection of $H$ on $AL$. Then $M_{BC}A_Y=M_{BC}A_{LX}$.
\end{theorem}

\begin{center}
\begin{figure}[htbp]
\includegraphics[scale=.4]{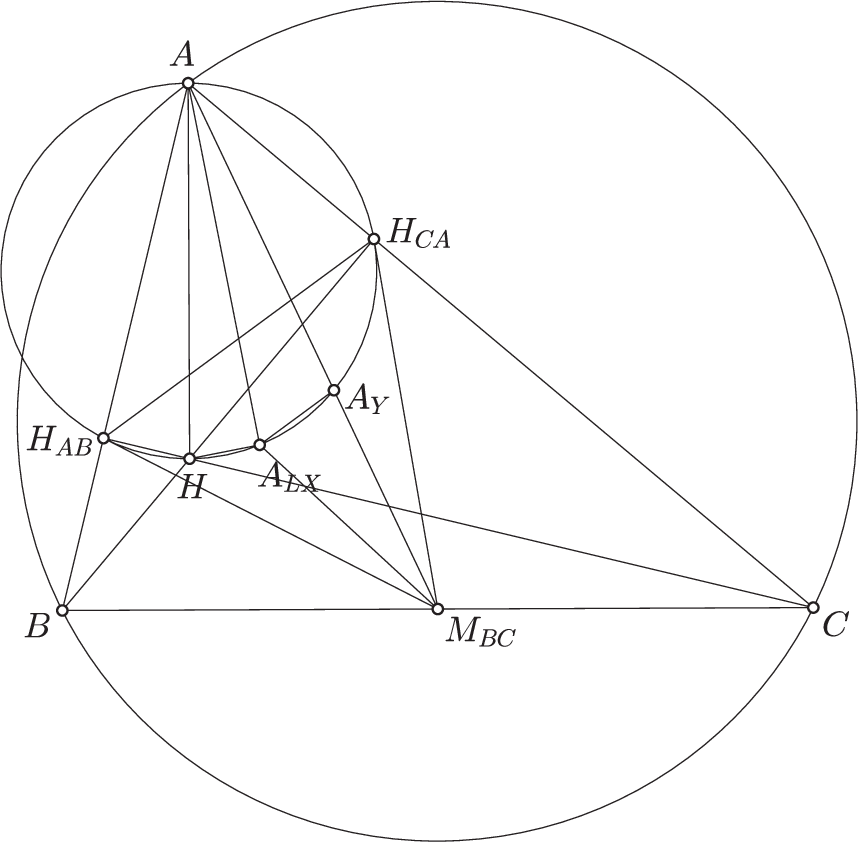}
\caption{Proof of Theorem \ref{theorem_equal}}
\end{figure}
\end{center}

\begin{proof} Draw altitudes $BH_{CA}$ and $CH_{AB}$ of triangle $ABC$. It can be seen that $A, H, H_{CA}, H_{AB}, A_Y, A_{LX}$ lie on the circle with diameter $AH$. Since $M_{BC}$ is the intersection of two tangent lines at $H_{CA}$ and $H_{AB}$ of $\left(AH\right)$ so $AH_{AB}A_YH_{BC}$ is a harmonic quadrilateral. Since $AL$ is a median line of triangle $AH_{CA}H_{AB}$ hence $H_{CA}H_{AB} \parallel A_YA_{LX}$. Moreover, as $M$ lies on the perpendicular bisector of $H_{CA}H_{AB}$, which is also the perpendicular bisector of $A_YA_{LX}$, so that $M_{BC}A_Y=M_{BC}A_{LX}$. \end{proof}

\begin{cor}
$\left(A_YM_{BC}A_{LX}\right)$ is tangent to the Euler circle of triangle $ABC$.
\end{cor}

\begin{theorem}
\label{theorem_symmedian}
Let $L_{BC}$ be intersection of $AL$ and $BC$. Then $A_YL_{BC}$ is a symmedian line of triangle $A_YBC$.
\end{theorem}

\begin{center}
\begin{figure}[htbp]
\includegraphics[scale=.5]{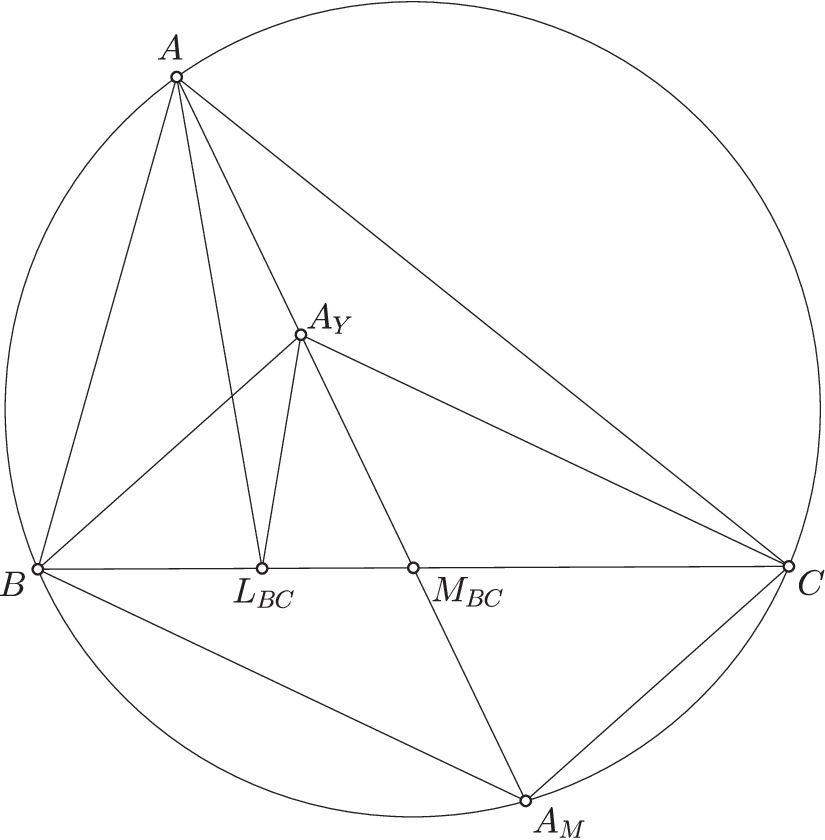}
\caption{Proof of Theorem \ref{theorem_symmedian}}
\end{figure}
\end{center}

\begin{proof}
Since $AL_{BC}$ is a symmedian line of triangle $ABC$ so
$$\frac{\overline{L_{BC}B}}{\overline{L_{BC}C}}=-\left(\frac{AB}{AC}\right)^2.$$
On the other hand,
$$\left(\frac{AB}{AC}\right)^2=\left(\frac{AB}{M_{BC}B}\cdot\frac{M_{BC}C}{AC}\right)^2=\left(\frac{A_MC}{A_MM_{BC}}\cdot\frac{A_MM_{BC}}{A_MB}\right)^2=\left(\frac{A_YB}{A_YC}\right)^2$$
So
$$\frac{\overline{L_{BC}B}}{\overline{L_{BC}C}}=-\left(\frac{A_YB}{A_YC}\right)^2.$$
We conclude that $A_YL_{BC}$ is a symmedian line of triangle $A_YBC$. \end{proof}

\begin{theorem}
\label{theorem_tangent}
$\left(A_YM_{BC}A_{LX}\right)$ is tangent to $\left(BHC\right)$.
\end{theorem}

\begin{lemma}
\label{lemma_tangent_circle}
Given triangle $ABC$. On segment $BC$, we choose $D$, $E$ such that $AD$, $AE$ are isogonal conjugate of each other in this triangle. Prove that $\left(ABC\right)$ is tangent to $\left(ADE\right)$.
\end{lemma}

\begin{center}
\begin{figure}[htbp]
\includegraphics[scale=.5]{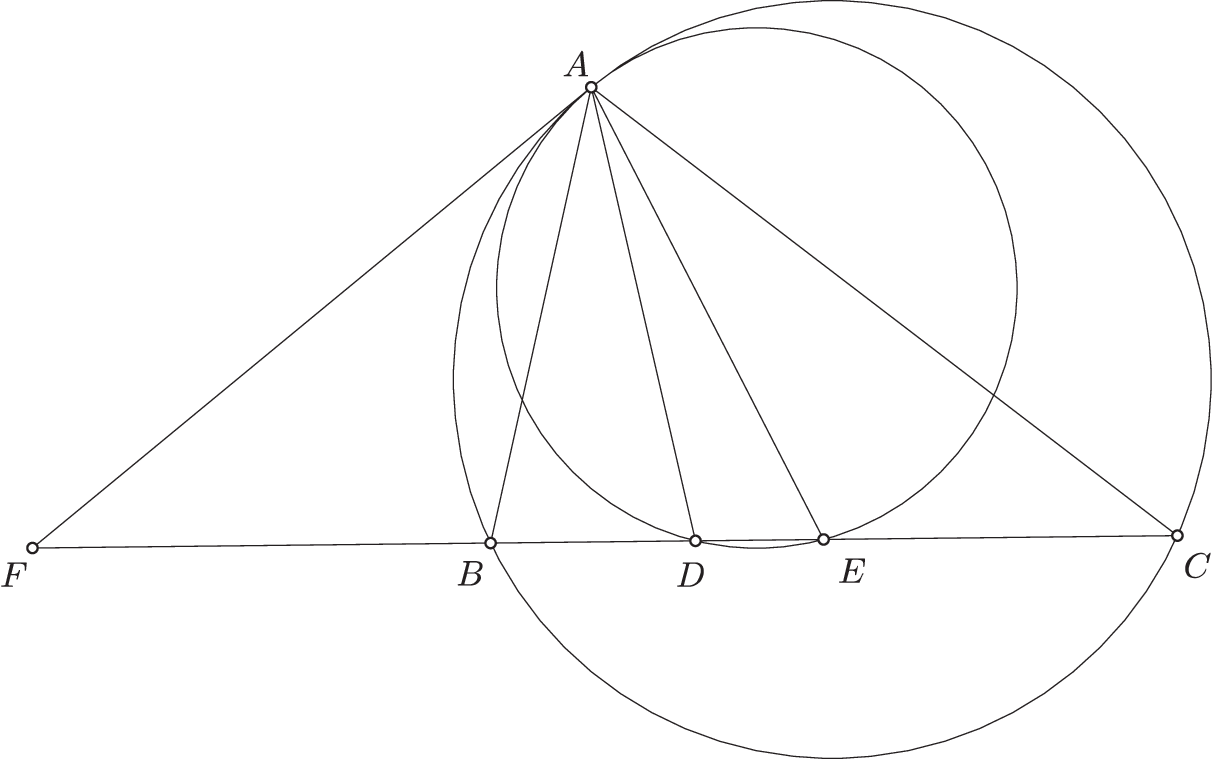}
\caption{Proof of Lemma \ref{lemma_tangent_circle}}
\end{figure}
\end{center}

\begin{proof}
The tangent line at $A$ of $\left(ABC\right)$ cuts line $BC$ at $F$. Then
$$(AF, AD)=(AF, AB)+(AB, AD)=(CA, CB)+(EA, CA)=(EA, CB)$$ 
so $AF$ is tangent to $\left(ADE\right)$. Thus, $\left(ABC\right)$ is tangent to $\left(ADE\right)$.
\end{proof}

Back to Theorem \ref{theorem_tangent},

\begin{center}
\begin{figure}[htbp]
\includegraphics[scale=.5]{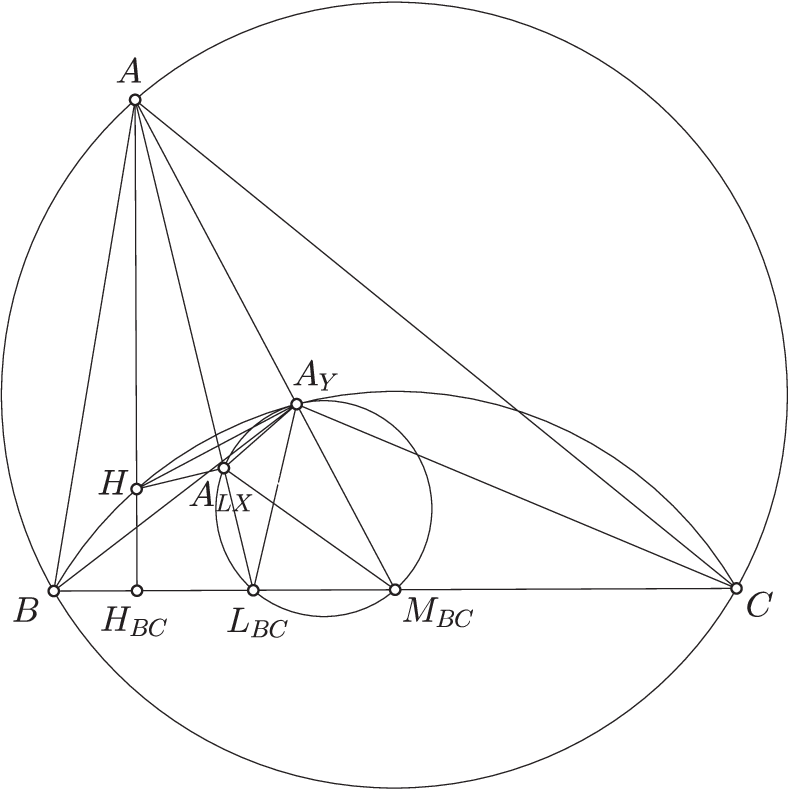}
\caption{Proof of Theorem \ref{theorem_tangent}}
\end{figure}
\end{center}

\begin{proof}
Draw the altitude $AH_{BC}$ of triangle $ABC$. We have that $H_{BC}$ and $A_{LX}$ lie on the circle with diameter $HL_{BC}$. Therefore,
$$(L_{BC}A_{LX}, L_{BC}M_{BC})=(HA_{LX}, HH_{BC})=(A_YA_{LX}, A_YA).$$
So $A_Y, A_{LX}, L_{BC}, M_{BC}$ are cyclic. Using Theorem \ref{theorem_symmedian}, $A_YL_{BC}$ is a symmedian line of triangle $A_YBC$. We use Lemma \ref{lemma_tangent_circle} to conclude that $\left(A_YM_{BC}A_{LX}\right)$ touches $\left(BHC\right)$.
\end{proof}

\begin{theorem}
\label{theorem_ABYACY_lie}
The rays $BA_Y$, $CA_Y$ cuts $\left(ABC\right)$ at $A_{BY}$, $A_{CY}$. Let $A'_{BY}, A'_{CY}$ be symmetry points of $A_{BY}, A_{CY}$ through $CA$, $AB$ respectively. Then $A'_{BY}$ and $A'_{CY}$ lie on $AL$.
\end{theorem}

\begin{center}
\begin{figure}[htbp]
\includegraphics[scale=.5]{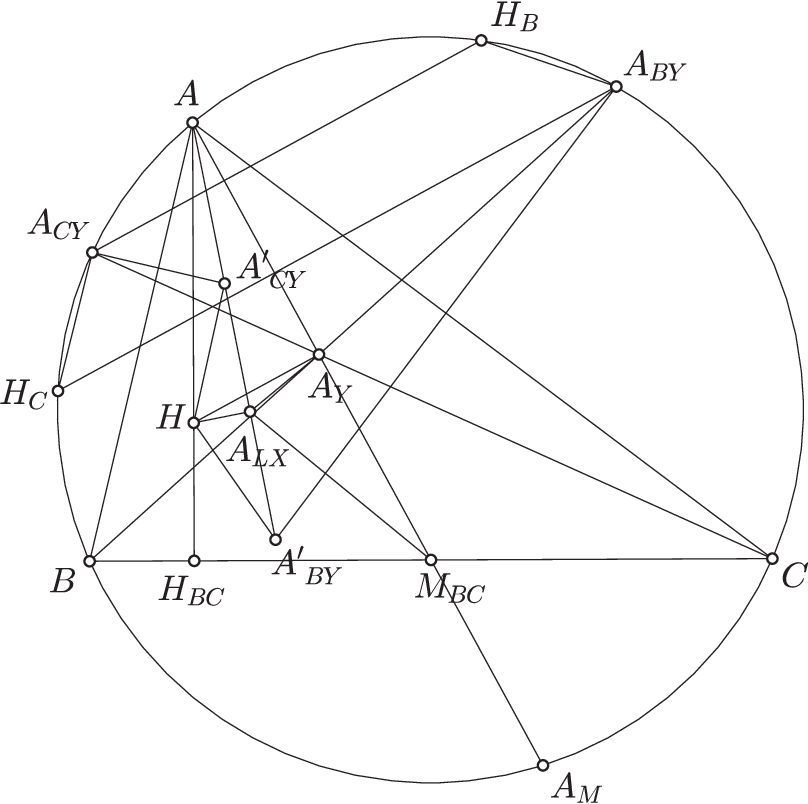}
\caption{Proof of Theorem \ref{theorem_ABYACY_lie}}
\end{figure}
\end{center}

\begin{proof}
We have
$$\begin{aligned} (AA'_{BY}, AC) & =(AC, AA_{BY})=(BC, BA_{BY}) \\ & =(CB, CA_M)=(AB, AA_M) \\ & =(AL, AC). \end{aligned}$$
Thus $AL$ passes $A'_{BY}$. Similarly, $AL$ passes $A'_{CY}$.
\end{proof}

\begin{theorem}
$A_{LX}$ is midpoint of $A'_{BY}A'_{CY}$.
\end{theorem}

\begin{proof}
Rays $BH$ and $CH$ cut $\left(ABC\right)$ respectively at $H_B$ and $H_C$. We have
$$(H_BA_{CY}, H_BB)=(CA_{CY}, CB)=(HA_Y, HH_B)$$
so $H_BA_{CY} \parallel HA_Y$, where $HA_Y \perp AG$. Simiarly $H_CA_{BY} \parallel HA_Y$, where $HA_Y \perp AG$ hence $H_BA_{CY}H_CA_{BY}$ is a isosceles trapezoid. We conclude that $H_CA_{CY}=H_BA_{BY}$. Using symmetry transformation through $AB$ and $AC$, we have $\mathbf{S}_{AB}\colon H, A_{CY} \mapsto H_C, A'_{CY}$ and $\mathbf{S}_{AC}\colon H, A_{BY} \mapsto H_B, A'_{BY}$. Therefore $HA'_{BY}=H_BA_{BY}=H_CA_{CY}=HA'_{CY}$, which means triangle $HA'_{BY}A'_{CY}$ is isosceles at $H$. Moreover, $A_{LX}$ is the projection of $H$ on $A'_{BY}A'_{CY}$ so $A_{LX}$ is midpoint of $A'_{BY}A'_{CY}$.
\end{proof}

\section{Conclusion}
This article is a discussion on the properties associated with a pair of distinct triangle centers. Further investigations can be conducted in this domain to delve deeper into the subject matter and uncover additional distinctive characteristics pertaining to triangles.

\bigskip

\bigskip

\bigskip

FACULTY OF MATHEMATICS AND COMPUTER SCIENCE, UNIVERSITY OF SCIENCE, VIETNAM NATIONAL UNIVERSITY HO CHI MINH CITY, 227 NGUYEN VAN CU, WARD 4, DISTRICT 5, HO CHI MINH CITY

\textit{E-mail address}: \texttt{21c29018@student.hcmus.edu.vn}
\end{document}